\documentclass[10pt,leqno,twoside]{amsart}

\usepackage{tikz, subfigure, xcolor} 
\usepackage{amsmath}
\usepackage{amsthm}
\usepackage{amssymb}
\usepackage{amsfonts}
\usepackage{cite}
\usepackage{dsfont}
\usepackage{hyperref}
\usepackage{comment}
\usepackage{inputenc}
\usepackage[normalem]{ulem}
\usepackage{doi}

\newtheorem{theorem}{Theorem}[section]
\newtheorem{lemma}[theorem]{Lemma}
\newtheorem{proposition}[theorem]{Proposition}
\newtheorem{corollary}[theorem]{Corollary}
\theoremstyle{definition}
\newtheorem{definition}[theorem]{Definition}
\newtheorem{remark}[theorem]{Remark}

\numberwithin{equation}{section}

\renewcommand{\div}{\mathrm{div}\,}				
\newcommand{\divh}{\mathrm{div}_{\!H}\,}		
\newcommand{\nablah}{\nabla_{\!H}\,} 			
\newcommand{\Deltah}{\Delta_{\!H}\,}  			
\newcommand{\lso}{L^2_{\overline{\sigma}}} 			
\newcommand{\ls}{L^2_{\sigma}} 			

\DeclareMathOperator{\nablatoalpha}{\nabla^{\alpha}}

\providecommand{\norm}[1]{\left\| #1 \right\|} 

\renewcommand{\d}{\operatorname{d}} 
\DeclareMathOperator{\dt}{\d\! t}		
\DeclareMathOperator{\ds}{\d\! s}		
\DeclareMathOperator{\ddt}{\frac{\d}{\dt}} 
\DeclareMathOperator{\tddt}{\tfrac{\d}{\dt}} 

\newcommand{\R}{\mathbb{R}}

\newcommand{\N}{\mathbb{N}}
\newcommand{\PP}{\mathbb{P}}

\title[Primitive Equations with Horizontal Viscosity]{Primitive Equations with Horizontal Viscosity: The Initial Value and the Time-Periodic Problem for Physical Boundary Conditions}



\author[Hussein]{Amru Hussein}
\address{Department of Mathematics,
	TU Kaiserslautern, Paul-Ehrlich-Stra{\ss}e 31,
	67663 Kaiserslautern, Germany}
\email{hussein@mathematik.uni-kl.de}

\author[Saal]{Martin Saal}
\address{Scuola Normale Superiore, Piazza dei Cavalieri 7, 56126 Pisa, Italy}
\email{martin.saal@sns.it}

\author[Wrona]{Marc Wrona}
\address{Departement of Mathematics,
	TU Darmstadt, Schlossgartenstr. 7, 64289 Darmstadt, Germany}
\email{wrona@mathematik.tu-darmstadt.de}

\begin{document}

\begin{abstract}
The $3D$-primitive equations with only horizontal viscosity are considered on a cylindrical domain $\Omega=(-h,h) \times G$, $G\subset \R^2$ smooth, with the physical Dirichlet boundary conditions on the sides. Instead of considering a vanishing vertical viscosity limit, we apply a direct approach which in particular avoids unnecessary boundary conditions on top and bottom. For the initial value problem, we obtain existence and uniqueness  of local $z$-weak solutions for initial data in $H^1((-h,h),L^2(G))$ and local strong solutions for initial data in $H^1(\Omega)$. 
If $v_0\in H^1((-h,h),L^2(G))$, $\partial_z v_0\in L^q(\Omega)$ for $q>2$, then the $z$-weak solution regularizes instantaneously and thus extends to a global strong solution.
This goes beyond the global well-posedness result by Cao, Li and Titi (J. Func. Anal. 272(11): 4606-4641, 2017) for initial data near $H^1$ in the periodic setting. For the time-periodic problem, existence and uniqueness of $z$-weak and strong time periodic solutions is proven for small forces. 
Since this is a model with hyperbolic and parabolic features for which classical results are not directly applicable, such results for the time-periodic problem even for small forces are not self-evident.
\end{abstract}

\maketitle


\section{Introduction and main results}
The $3D$-primitive equations are one of the fundamental models for geophysical flows, and they are used for describing oceanic and atmospheric dynamics. They are derived from the Navier-Stokes equations assuming a hydrostatic balance.
The subject of this work are the initial value and time-periodic problem for the primitive equations with only horizontal viscosity and the physical lateral Dirichlet boundary conditions.

The motivation to study this problem is that in many geophysical models the horizontal viscosity is considered to be dominant and the vertical viscosity is neglected. From the analytical point of view such models with only partial viscosity terms are also very interesting since they combine features of both parabolic diffusion equations in horizontal directions represented  by the term $-\Deltah$ and hyperbolic transport equations in vertical direction represented by the term $w \partial_z v$, compare \eqref{eq:primeqhorvisc} below. Roughly speaking, one thus expects that regularity is preserved in the vertical direction while it is smoothed in the horizontal directions. Following this intuition allows us to identify classes of initial data for which this problem is locally or even globally well-posed. 

Many forces acting on geophysical flows such as the attraction by the moon, which becomes visible in the falling and rising tides, are time-periodic. Moreover, in some models the wind is described as a perturbation of a periodic function. 
A time-periodic force adds in each period energy to the system, and since there is only partial viscosity it is not self-evident whether the system remains stable enough to have time-periodic solutions. However, it turns out that at least for forces being small over one period of time, there are unique small time-periodic solutions. 

This work  is part of the third author's PhD thesis \cite{Wrona2020}, and therein some of the computations are elaborated in more detail.

\subsection{Primitive equations with only horizontal viscosity}
To be precise, here time intervals $(0,{T})$ for ${T}\in (0,\infty)$ and a cylindrical domain $\Omega$ are considered, 
\begin{align}\label{eq:Omega}
\Omega :=  (-h, h) \times  G \subset \R^3, \quad \hbox{for} \quad h > 0 \quad \hbox{and a smooth domain} \quad G \subset \R^2,
\end{align}
where
the boundary $\partial \Omega$ decomposes into a lateral, upper and bottom part
\begin{align*}
\Gamma_l:= \partial G\times [-h,h], \quad \Gamma_u:=  G \times \{+h\}, \quad  \Gamma_b:=  G \times \{-h\}.
\end{align*}
The primitive equations  
describe the velocity $u=(v,w)\colon \Omega \rightarrow \R^3$ and the pressure $p\colon \Omega \rightarrow\R$ of a fluid, where $v = (v_1, v_2)$ denotes the
horizontal components and $w$ stands for the vertical one. The primitive equations with horizontal viscosity are
\begin{align}
\left\{
\begin{array}{rll}
  \partial_t v + v \cdot \nablah v + w \partial_z v
  - \Deltah v + \nablah p&=f, & \text{ in } \Omega
  \times (0, {T}),\\
   \partial_z p&=0, &\text{ in } \Omega \times (0, {T}), \\
  \divh v + \partial_z w&=0, & \text{ in } \Omega \times
  (0, {T}), \\
  v(t=0)&=v_0, & \text{ in } \Omega
\end{array}\right. \label{eq:primeqhorvisc}
\end{align}
which are supplemented by the boundary conditions
\begin{align}\label{eq:bc}
v=0 \quad \hbox{on }\Gamma_l \times (0,{T})\quad \hbox{and} \quad w =0 \quad \hbox{on} \quad \Gamma_u\cup \Gamma_b \times (0,{T}).
\end{align}
The first boundary condition is a lateral no-slip boundary condition and the latter is due to the divergence free condition $\div u =0$ and $\nu\cdot u =0$ for the outer normal derivative $\nu$ on $\partial \Omega$.
Here, 
$x,y\in G$ are the horizontal coordinates and $z\in (-h,h)$ the vertical coordinate, 
$\nablah =(\partial_x, \partial_y)^T$, $\divh = \nablah^{\ast}$ and $\Deltah=\partial_x^2+\partial_y^2$ denote the horizontal
gradient, divergence and Laplacian, respectively and 
$v \cdot \nablah=v_1\partial_x+v_2\partial_y$. 
Note, that for the primitive equations the nonlinear term $w\partial_z v$ is stronger 
compared to the nonlinearity of the Navier-Stokes equation since $w=w(v)$ given by \eqref{eq:w} below involves first order derivatives, while the pressure here is only two-dimensional.

For simplicity we have formulated the equations without the Coriolis force, 
but being a zero order term it does not alter the well-posedness results 
discussed here. Moreover, we consider only the velocity equation without temperature or salinity focusing on the mathematical difficulties.  
The general anisotropic primitive equations are given if one replaces in \eqref{eq:primeqhorvisc} the term $\Deltah$ by $\nu_1\Deltah + \nu_2\partial_{zz}$ for  horizontal viscosity $\nu_1\geq 0$ and vertical viscosity $\nu_2\geq 0$. Here, physical constants are normalized to one, thus we consider the case $\nu_1=1$ and $\nu_2=0$.

\subsection{Previous results} 
Cao, Li and Titi, see \cite{CaoLiTiti2016, CaoLiTiti2017}, have been the first to study the primitive equations with only horizontal viscosity analytically. 
They tackled this problem in a periodical setting by considering a vanishing vertical viscosity limit, i.e., 
\begin{align*}
-\Deltah - \varepsilon
\partial_z^2 \quad \hbox{for}\quad  \varepsilon\to 0,
\end{align*}
and by this strategy they obtained remarkable global strong well-posedness results for the initial value problem  for initial data with regularity near $H^1(\Omega)$, and local well-posedness for initial data in $H^1(\Omega)$.  
Recently, the second author \cite{Saal} applied a more direct approach considering the system without vanishing viscosity limit. Thereby local well-posedness results even for less partial viscosities has been proven, and for only horizontal viscosity unnecessary boundary conditions on bottom and top have been avoided. 

Note that for the Navier-Stokes equations with only horizontal viscosity there are also some local well-posedness results, cf. \cite[Chapter 6]{Danchin}. 

The mathematical 
analysis of the initial value problem for the primitive equations with full viscosity, i.e., with viscosity term $\nu_1\Deltah + \nu_2\partial_{zz}$ where $\nu_1,\nu_2>0$, has been started by Lions, 
Temam and Wang \cite{Lionsetal1992, Lionsetal1992_b, Lionsetal1993} which launched a lot of activity in the analysis of these equations.
In difference to the $3D$ Navier-Stokes equations 
the primitive equations are known to be time-global well-posed 
for initial data in $H^1(\Omega)$ by the breakthrough result 
of Cao and Titi \cite{CaoTiti2007}, see also \cite{Ziane2007} for different boundary conditions and non-cylindrical domains. Refinements of this include global well-posendess for initial data with $v_0,\partial_z v_0\in L^2(\Omega)$, see \cite{Ju2017}, or $v_0\in L^1((-h,h),L^\infty(G))$, see \cite{GigaGriesHusseinHieberKashiwabara2017NN}.  

For the inviscid $3D$-primitive equations, i.e., $\nu_1=\nu_2=0$, blow-up results are known by Wong \cite{Wong2015}, see also \cite{Caoetal2015}, and there are ill-posedness results for Sobolev spaces by Han-Kwan and Nguyen \cite{HanNguyen2016}. Local well-posedness has been proven only for analytical data by Kukavica et al.
\cite{Kukavica2011}. The primitive equations with partial viscosity are an intermediate model between these well- and ill-posed situations.    

For more information on previous results on the primitive
equations we refer to the works of Washington and
Parkinson \cite{WashingtonParkinson1986}, Pedlosky
\cite{Pedlosky1987}, Majda \cite{Majda2003} and Vallis
\cite{Vallis2006}; see also the recent surveys by Li and Titi 
\cite{LiTiti2016} and by Hieber and the first author \cite{HH20}. 

\subsection{Main results and discussion}
Our main results are stated in the following.
Below, the notions of weak and $z$-weak solutions are made precise in Definitions~\ref{def:weaksolution} and~\ref{def:zweaksolution}, where function spaces are introduced in Subsection~\ref{subsec:functionspaces}.
\begin{theorem}[Local solutions for the initial value problem]\label{thm:IWP_loc} 
\ \ \ \ \ \ \ \ \ \ \ \ \ \ \ \ \ \ 
\begin{enumerate}
	\item[(a)] 
	Let $f\in  L^2((0,{T}),H^1((-h,h),H^{-1}(G))^2)$ and $v_0\in H^1((-h,h),L^2(G))^2$ with $\divh \overline{v_0}=0$. Then there exists a time $T'\in (0,T]$ and a unique $z$-weak solution to the initial boundary value problem 
	\eqref{eq:primeqhorvisc}, \eqref{eq:bc}
	on $(0,T')$, i.e., a weak solution with
	\begin{align*}
		v_z\in L^{\infty}((0,T'),L^2(\Omega)^2),   \nablah v_z \in L^2((0,T'),L^2(\Omega)^{2\times 2}),
	\end{align*}
and this $z$-weak solution satisfies  $v\in C^0([0,T'],L^2(\Omega)^2)$.	
	One has $T'=T$ if $\|v_0\|_{H^1((-h,h),L^2(\Omega))}$ and $\|f\|_{L^2((0,T),H^1((-h,h),H^{-1}(G)))}$ are sufficiently small.
	\item[(b)] 
	If $v_0\in H^1(\Omega)^2$ with $v_0(\cdot, z)\vert_{\partial G}=0$ for almost all $z\in (-h,h)$,  $\divh \overline{v_0}=0$, and $f\in   L^2((0,{T}),H^1((-h,h),L^{2}(G))^2)$, then there exists a time $T'\in (0,T]$ and a unique strong solution to \eqref{eq:primeqhorvisc}, \eqref{eq:bc} 
	on $(0,T')$, i.e. a $z$-weak solution where in addition
	\begin{align*}
	\nablah v \in L^{\infty}((0,T),L^2(\Omega)^{2\times 2}), \quad 
	\Deltah v, \partial_t v \in L^2((0,T),L^2(\Omega)^2).
	\end{align*} 
	For  $\|v_0\|_{H^1}$ and $\|f\|_{L^2((0,T),H^1((-h,h),L^2(G)))}$ sufficiently small, one has \\$T'=T$.
\end{enumerate}
\end{theorem}

\begin{remark}\label{rem:contdep}
	Continuous dependence on the data can be proven as well by adapting the estimates obtained in the proof of Theorem~\ref{thm:IWP_loc}.
\end{remark}

\begin{theorem}[Global solutions for the initial value problem]\label{thm:IWP_glob} 
\ \ \ \ \ \ \ \ \ \ \ \ \ \ \ \ \ \ 
\begin{enumerate} 
		\item[(a)] Let $v_0\in H^1_\eta(\Omega)^2$  with $v_0(\cdot, z)\vert_{\partial G}=0$ for almost all $z\in (-h,h)$, $\divh \overline{v_0}=0$ and $f\equiv 0$, where
	\begin{align*}
	H^1_\eta(\Omega):=\{v \colon \Vert v_0\Vert _{H^1_\eta}=\Vert v_0\Vert_{H^1}+\Vert\partial_zv_0\Vert _{2+\eta}+\Vert v_0\Vert _\infty <\infty\},
	\end{align*}
	for some $\eta >0$, then there exists a unique global strong solution to \eqref{eq:primeqhorvisc}, \eqref{eq:bc} on $(0,T)$ for any $T>0$. Moreover,
	\[\sup_{t\in [0,{T}]}(\Vert \nabla v\Vert _2^2+\Vert \partial _zv\Vert _{2+\eta}^{2+\eta})+\int_0^{T}\Vert \nabla_H\nabla v\Vert _2^2\le C_{\eta ,h, T}(\Vert v_0\Vert _{H^1_\eta}),
		\] for an increasing function $C_{\eta ,h, T}$ depending on $\eta $, $h$, $T$.
\item[(b)] Let $v_0\in H^{1}((-h,h),L^2(G)^2)$ with $\partial_z v_0\in L^{2+\eta}(\Omega)^2$ for $\eta>0$ with $\divh \overline{v_0}=0$ and $f\equiv 0$. Then the unique $z$-weak solution from Theorem~\ref{thm:IWP_loc} (a) extends to a unique global strong solution to \eqref{eq:primeqhorvisc}, \eqref{eq:bc} on $(0,T)$ for any $T>0$. Moreover, for any $\delta \in (0,T)$
	\[\sup_{t\in [\delta,{T}]}(\Vert \nabla v\Vert _2^2+\Vert \partial _zv\Vert _{2+\eta}^{2+\eta})+\int_{\delta}^{T}\Vert \nabla_H\nabla v\Vert _2^2\le C_{T,h,\delta,v_0}
\] for a constant $C_{T,h,\delta, v_0}$, depending on $T,h,\delta, v_0$. 
\end{enumerate}
\end{theorem}

Note that the regularity of the initial value in Theorem~\ref{thm:IWP_loc} (a) is similar to the one obtained for the Navier-Stokes equation with horizontal viscosity, compare \cite[Theorem 6.2]{Danchin}. It is also the same condition obtained by Ju for the existence and uniqueness of global $z$-weak solutions for the  primitive equations with full viscosity, see \cite{Ju2017}. 

Theorem~\ref{thm:IWP_loc} (b) and Theorem~\ref{thm:IWP_glob} (a) correspond to the result by Cao, Li and Titi in \cite[Theorem 1.1]{CaoLiTiti2017}. However, they consider a cubical domain with periodic boundary conditions in all three directions. As already pointed out in \cite{Saal}, vertical boundary conditions are not necessary, but they are preserved by the equation. Here, we consider the more physical Dirichlet boundary conditions on the sides and no boundary condition on top and bottom. 

The proof of the \textit{a priori} bounds in \cite{CaoLiTiti2017} uses a vanishing vertical viscosity limit. Here, we follow a more direct approach considering the case of horizontal viscosity without such limits. For the global \textit{a priori} bound, we have been able to adapt the overall strategy of Cao, Li and Titi, but due to the boundary conditions here, controlling the pressure terms becomes more involved. Note also that in \cite{CaoLiTiti2017} the \textit{a priori} bound is proven for periodic boundary conditions in all three directions, here we do not require any boundary conditions on the top and bottom part of the boundary.

Moreover, using that regularity is preserved in the vertical directions while being smoothed in the horizontal directions, we show in Theorem~\ref{thm:IWP_glob} (b) that a $z$-weak solution with slightly more integrability of the initial data  regularizes to reach the setting of Theorem~\ref{thm:IWP_glob} (a) for $t>0$. Thus existence and uniqueness of global strong solutions holds even for a larger class of initial conditions. It is remarkable that the regularity of the initial conditions required in  \ref{thm:IWP_glob} (b) is very close to the one obtained by Ju for the case of full viscosity, cf. \cite{Ju2017}.

Furthermore, for the local well-posedness results a force term is included here which allows us to analyze the time-periodic problem.     
The notion of $T$-periodic $z$-weak solutions is explained in Definition~\ref{def:waekper} below.
\begin{theorem}[Time-periodic problem]\label{thm:periodic}
There  exists $\varepsilon>0$ 		
such that
	\begin{enumerate}
		\item[(a)] if $f\in  L^2((0,T),H^1((-h,h),H^{-1}(G))^2)$ and $\|f\|_{ L^2((0,T),H^1((-h,h),H^{-1}(G))}\!<\varepsilon$, then there exists a unique $z$-weak $T$-periodic solution $v$,  where $v$ lies in the regularity class given in Theorem~\ref{thm:IWP_loc} $(a)$; 
		\item[(b)] if $f\in  L^2((0,T),H^1((-h,h),L^2(G))^2)$ with $\|f\|_{ L^2((0,T),H^1((-h,h),L^2(G)))}<\varepsilon$, then  there exists a unique strong $T$-periodic solution $v$, i.e., $v$ lies in the regularity class given in Theorem~\ref{thm:IWP_loc} $(b)$. 
	\end{enumerate}
\end{theorem}
Here, existence of a $T$-periodic $z$-weak solution means that there exists a $v_0$ such that there exists a  $z$-weak solution to the initial boundary value problem \eqref{eq:primeqhorvisc},  \eqref{eq:bc} with initial condition $v_0$ and force $f$ which is $T$-periodic.

It seems that so far, there has been no result on the time-periodic problem for partial viscosities. For the primitive equations with full viscosity there are several results, see Hsia and Shiue \cite{Hsia2013} and Tachim Medjo \cite{Tachim2010}, on the existence of unique global strong time-periodic solutions for periodic forces assuming a smallness condition on the force. In contrast, in \cite{Galdi2017} existence  of strong periodic solutions for possibly large periodic forces has been shown, but the solutions are possibly non-unique. 

Here, we have adapted the strategy by Galdi, Hieber and Kashiwabara in \cite{Galdi2017} to consider the Poincar\'{e} map for the construction on time-periodic solutions where we take advantage of the \textit{a priori} estimates obtained for the initial value problem. 
A crucial ingredient in our proof is that due to the lateral Dirichlet boundary conditions, there is a Poincar\'e inequality of the type
\begin{align*}
\|v\|_{L^2(\Omega)} \leq C \|\nablah v\|_{L^2(\Omega)}.
\end{align*}

Note that for parabolic problems there are quite a few results on time-periodic solutions. For instance \L ukaszewicz et al. \cite{Lukaszewicz} treat the the case certain of semilinear parabolic equations in Hilbert spaces. There are also the maximal $L^p$-regularity approaches in Banach spaces for time-periodic solutions by Geissert et al. \cite{Huy}, and by 
Kyed and co-authors, cf. \cite{Kyed2017, Eiter2017, Celik2018} and the references therein. More concretely, for the Navier-Stokes equations there are also many results on periodic solutions going back to the work of Serrin \cite{Serrin}, see e.g. also \cite{Morimoto} and references therein for a more recent survey. We would like to emphasis that for the case of partial viscosity considered here the system is not purely parabolic anymore, and in particular since the vertical derivatives in the non-linearity cannot be controlled by the linear part all these approaches are not applicable. Instead one has to extract additional information for these particular equations.

\subsection{Organization of the paper} In the subsequent Section~\ref{sec:pre} basic definitions and notations are introduced. In particular, the spaces of hydrostatic-solenoidal functions and the notions of weak and $z$-weak solutions to the initial value and time-periodic problem and their regularity properties are discussed. In Section~\ref{sec:loc} the existence and uniqueness of local $z$-weak  and local strong solutions is proven, respectively. The time-periodic problem is discussed in Section~\ref{sec:per} including the proof of Theorem~\ref{thm:periodic}. In Section~\ref{sec:glob}, global \textit{a priori} bounds are proven, and the proof of Theorem~\ref{thm:IWP_glob} is given. 
Some auxiliary results are collected in Section~\ref{sec:ineq}.

\section{Preliminaries}  \label{sec:pre}
\subsection{Function spaces and notations}\label{subsec:functionspaces}
By $L^2(\Omega)$ we denote the standard real Lebesgue space with  scalar product
\begin{align*}
\left< f,g \right>_{\Omega}:=\int_{\Omega}f(x,y,z)g(x,y,z)\d(x,y,z),
\end{align*}
where $L^2(G)$ and $\left< f,g \right>_{G}$ are defined analogously. By $\norm{f}_{L^2(\Omega)}$ and 
$\norm{f}_{L^2(G)}$ we denote the induced norm dropping the subscripts $\Omega$ and 
$G$ in the notation if there is no ambiguity. 
For $f,g\in L^2((0,T),L^2(\Omega))$ we write
\begin{align*}
\left< f,g \right>_{\Omega,T}:=\int_{0}^T\int_{\Omega}f(t,x,y,z)g(t,x,y,z)\d(x,y,z)\d\! t
\end{align*}
for the scalar product in space and time.
For a function $f\in L^{\infty}((0,T), L^{\infty}(\Omega))$ we use the abbreviation
\begin{align*}
\norm{f}_{\infty}:=\sup_{t\in(0,T)}\norm{f(t)}_{L^{\infty}(\Omega)}.
\end{align*}

For $s\in\N$ the space $H^s(\Omega)$ consists of $f\in L^2(\Omega)$ such that 
$\nablatoalpha f \in L^2(\Omega)$ for $|\alpha|\leq s$ endowed with the norm
\begin{align*}
\norm{f}_{H^s(\Omega)}=\big(\sum_{|\alpha|\leq s} \norm{\nablatoalpha f}_{L^2(\Omega)}^2 \big)^{1/2},
\end{align*}
and $H^s_0(\Omega):=\{f\in H^s(\Omega)\colon f\vert_{\partial\Omega} =0 \}$.
Here we used the multi-index notation 
$\nablatoalpha=\partial_x^{\alpha_1} 
\partial_y^{\alpha_2}\partial_z^{\alpha_3}$ 
for $\alpha\in\N_0^3$. The spaces $H^s(G)$ and $H_0^s(G)$
are defined analogously, and we will again 
just write $\norm{f}_{H^s}$ if there is no 
ambiguity.
For non-integer $s\geq 0$, the spaces $H^s$
are defined by complex interpolation, and one sets by duality $H^{-s}=(H_0^{s})^{\prime}$, compare also \cite[Chapter 3]{Triebel}. 
Moreover, we set 
\begin{align}\label{eq:H01}
H_{0,l}^1(\Omega):=\{f\in H^1(\Omega)\colon f\vert_{\Gamma_l}=0\}
\end{align}
and $H_{l}^{-1}(\Omega):= H_{0,l}^1(\Omega)'$ is its dual space. Analogous definitions hold for Sobolev spaces $H^{s,p}$ for $p\in [1,\infty]$, where $H^s=H^{s,2}$, and for Sobolev spaces of functions with values in Banach spaces such as $H^1((0,T),L^p(\Omega))$ and $H^{s,q}((-h,h),L^{p}(G))$. Sometimes we use the short hand notation $H^{s,p}_zH^{r,q}_{xy}$ for $H^{s,p}((-h,h),H^{r,q}(G))$.

\subsection{Hydrostatic-solenoidal vector fields}\label{subsec:lso}
Now let us reformulate the primitive 
equations \eqref{eq:primeqhorvisc} and \eqref{eq:bc}. The divergence free 
condition $\partial_z w+\divh v=0$ and the 
boundary condition $w(z=\pm h)=0$ are equivalent to
\begin{align}\label{eq:w}
w(t,x,y,z) =-\divh \int_{-h}^{z}v(t,x,y,\xi) \d\! \xi \mbox{ and } 
\divh \int_{-h}^h v(t,x,y,\xi)\d\! \xi=0
\end{align}
for $v$ sufficiently smooth, e.g., $\divh v \in L^1(\Omega)$.
This means, that $\overline{v}$ -- the mean value of $v$ in the vertical direction -- is divergence free, i.e., $\divh \overline{v}=0$, where the vertical average and its complement are
\begin{align}\label{eq:defvbar} 
\overline v(t,x,y):=\frac{1}{2h}\int_{-h}^h v(t,x,y,\xi)  \d\! \xi \quad \hbox{and} \quad \tilde v:=v-\overline v.
\end{align}
Hence one identifies a suitable hydrostatic-solenoidal space as
\begin{align*}
\lso(\Omega):= \overline{\{v\in C_c^{\infty}(\Omega)^2\colon \divh \overline{v}=0  \}}^{||\cdot||_{L^2}},
\end{align*}
where $C_c^{\infty}$ stands for smooth compactly supported functions. Note that this space admits the decomposition
\begin{align}\label{eq:lso_split}
\lso(\Omega) =   \{v\in L^2(\Omega)^2\colon \overline{v}=0\} \oplus \ls(G),
\end{align}
and the hydrostatic Helmholtz projection thereon is
\begin{align}
\PP\colon L^2(\Omega)^2 \rightarrow \lso(\Omega), \quad \PP v = \tilde{v} + \PP_G \overline{v},
\end{align}
where $\ls(G)= \overline{\{\overline{v}\in C_c^{\infty}(G)^2\colon \divh \overline{v}=0  \}}^{||\cdot||_{L^2}}$ is the space of solenoidal vector fields over $G$, and $\PP_G$ the corresponding (classical) Helmholtz projection. More precisely, since due to the product structure $L^2(\Omega)=\overline{L^2(G) \otimes L^2(-h,h)}$, one obtains by applying $\PP$ that
\begin{align}\label{eq:lso_split2}
\lso(\Omega) = \overline{L^2(G)^2 \otimes L_0^2(-h,h)} \oplus \overline{\ls(G) \otimes \operatorname{span}\{1\}},
\end{align}
where $L^2_0(-h,h)=\{v\in L^2(-h,h)\colon \int_{-h}^{h} v(z) dz=0\}$ and $1\in L^2(-h,h)$ is a constant function. 

\subsection{Weak and $z$-weak solutions}\label{subsec:weak}
Next we give a precise notion of weak solutions.

\begin{definition}[Weak solution]\label{def:weaksolution}
Let $f\in L^2((0,T),L^2((-h,h),H^{-1}(G))^2)$ and  $v_0\in \lso(\Omega)$. A 
function $v$ is called \textit{a weak solution} of 
the primitive equations \eqref{eq:primeqhorvisc} with boundary conditions \eqref{eq:bc}
on $(0,T)$ with initial condition $v_0$ and force $f$ if
	\begin{enumerate}
		\item[(i)] One has that $v\colon [0,T]\rightarrow \lso(\Omega)$ is weakly continuous with $v(0)=v_0$ and 
		$$
		v\in L^{\infty}((0,T),\lso(\Omega)), \quad  \nablah v \in L^2((0,T),L^2(\Omega)^{2\times 2})
		$$
	with $v(t,z,\cdot,\cdot)\in H^1_0(G)^2$ almost everywhere for $z\in (-h,h)$;
		\item[(ii)]  For some constant $c>0$ it satisfies
		\begin{multline*}
		\left\|v\right\|_{L^{\infty}((0,T),L^2(\Omega))}
		+\left\|\nablah v\right\|_{L^2((0,T),L^2(\Omega))}\\
		\leq c \left( \left\|v_0\right\|_{L^2(\Omega)}+\left\|f\right\|_{L^{2}((0,T),H^{-1}(G))}\right);
		\end{multline*}
	\item[(iii)] 	$v$ satisfies \eqref{eq:primeqhorvisc} and \eqref{eq:bc} in the weak sense, i.e.,
		\begin{multline*}
		-\left< v, \varphi_t \right>_{\Omega,T} +\left<v(T),\varphi(T)\right>_{\Omega}  -\left<v_0,\varphi(0)\right>_{\Omega} 
		+\left< \nablah v, \nablah \varphi \right>_{\Omega,T} \\
		+\left< v\cdot \nablah v, \varphi \right>_{\Omega,T}-\left< w v, \varphi_z \right>_{\Omega,T}
		-\left< w_z v, \varphi \right>_{\Omega,T} 
		=\left< f, \varphi \right>_{\Omega,T}, 
		\end{multline*}
		where $w=w(v)$ is given by \eqref{eq:w}, holds for any 
		\begin{multline*}
			\varphi\in H^{1,1}((0,T),\lso(\Omega ))\cap C^0([0,T],\lso(\Omega )) \\ \cap L^2((0,T),L^2((-h,h),H_0^1(G))^2\cap H^1((-h,h),L^\infty (G))^2).
		\end{multline*}
	\end{enumerate}
\end{definition}

Note that there are different notions of weak solutions for the primitive equations, compare \cite{Galdi2017} or \cite{Tachim2010}. The notion of $z$-weak solutions for the primitive equations  has been introduced by Bresch et al. \cite{Bresch2003} as \textit{vorticity solutions} for the $2D$-case. It plays also an important role in the study of the $3D$-case with full viscosity, see \cite{Ju2017} and the references therein. This is adapted here to the case of only horizontal viscosity.
\begin{definition}[$z$-weak solution]\label{def:zweaksolution}
Let $f\in L^2((0,T),H^1((-h,h),H^{-1}(G))^2)$ and $v_0\in \lso(\Omega)$ with $\partial_z v_0\in L^2(\Omega)^2$. 
A weak solution $v$ 
of 
 the primitive equations \eqref{eq:primeqhorvisc} with boundary conditions \eqref{eq:bc}
 on $(0,T)$ with initial condition $v_0$ and force $f$ is called a \textit{$z$-weak solution} if 
additionally
\begin{align*}
v_z\in L^{\infty}((0,T),L^2(\Omega)^2)\quad \hbox{and} \quad \nablah v_z \in L^2((0,T),L^2(\Omega)^{2\times 2}).
\end{align*}
\end{definition}

\begin{definition}[$T$-periodic $z$-weak solutions]\label{def:waekper}
	A $z$-weak solution $v$ is called \textit{$T$-periodic} if $v(0)=v(T)$. 
\end{definition}

\begin{remark}\label{rem:nonlinterms}
\begin{itemize}
	\item[(a)] $z$-weak solutions are additionally in $C^0([0,T],L^2(\Omega))$, in this sense $v(0)=v(T)$ has to be understood in Definition \ref{def:waekper}.
 \item[(b)] For a weak solution one has for the non-linear terms 
 $w\cdot v,w_z\cdot v,v\cdot \nablah v\in L^2((0,T),L^1(\Omega)^2)$, and this guarantees that each term in the weak formulation is well-defined.	
A $z$-weak solution is regular enough to assure that even $w\cdot v_z\in L^2((0,T),L^1(\Omega)^2)$ 
and therefore $-\left< w v, \varphi_z \right>_{\Omega,T} -\left< w_z v, \varphi \right>_{\Omega,T} 
=\left< w v_z, \varphi \right>_{\Omega,T}$ for test functions $\varphi$ as in Definition~\ref{def:weaksolution}.
\end{itemize}
\end{remark}

\subsection{Regularity of $z$-weak solutions}
In the following proposition we show that for 
$z$-weak solutions the class of admissible test functions for which especially the nonlinear terms are well-defined is 
much larger than for weak solutions. This turns out to be useful when testing a $z$-weak solution with itself, and in particular
when proving the uniqueness of $z$-weak solutions. 

\begin{proposition}[Class of test functions for $z$-weak solutions]\label{lem:testfunctions}
	For \\ $f\in L^2((0,T),H^1((-h,h),H^{-1}(G))^2)$, $v_0\in \lso(\Omega)$ with $\partial_z v_0\in L^2(\Omega)$ 
	let $v$ be a $z$-weak solution to \eqref{eq:primeqhorvisc}, \eqref{eq:bc} on $(0,T)$. Then 
		\begin{align*}
		\left< \nablah v, \nablah \varphi\right>_{\Omega,T}, \quad  \left< v\cdot \nablah v, \varphi \right>_{\Omega,T}, \left< w \partial_z v , \varphi\right>_{\Omega,T}, \quad  \hbox{and } \left< f, \varphi\right>_{\Omega,T}		
		\end{align*}
are well-defined for all
		\begin{align*}
		\varphi \in L^{\infty}((0,T),\lso(\Omega )) \cap L^2((0,T),L^2((-h,h),H^1_0(G))^2).
		\end{align*}
\end{proposition}

\begin{proof}
	Let $(\varphi^{(n)})_n$ be a sequence of smooth functions such that $\varphi^{(n)}\to \varphi$ for $n\to\infty$ in $ L^{\infty}((0,T),\lso(\Omega )) \cap L^2((0,T),L^2((-h,h),H^1_0(G))^2)$. Then we 
	have 
	\begin{align*}
	\left< \nablah v, \nablah \varphi^{(n)} \right>_{\Omega,T} \to \left< \nablah v, \nablah \varphi\right>_{\Omega,T} \quad\text{and} \quad 
	\left< f, \varphi^{(n)} \right>_{\Omega,T} \to \left< f, \varphi\right>_{\Omega,T}
	\end{align*}
	for $n\to\infty$.\\
	The nonlinear terms $\left< w v, \varphi^{(n)}_z \right>_{\Omega,T} +\left< \partial_z w  v, \varphi^{(n)} \right>_{\Omega,T}=
	-\left< w \partial_z v , \varphi^{(n)} \right>_{\Omega,T}$ 
	and $\left< v\cdot \nablah v, \varphi^{(n)} \right>_{\Omega,T}$ have to be to handled with more care. Using Lemma~\ref{lemma:lowreg}~a) 
	with $f=v(t)$, $g=\varphi^{(n)}(t)$ and $h=\nablah v(t)$ we obtain
	\begin{align*}
	&\int_0^T \left|  \left< v(t)\cdot \nablah v(t), \varphi^{(n)}(t) \right>_{\Omega}\right|\dt \\
	&\qquad\leq c \int_0^T \norm{\nablah v(t)}^{1/2}_{L^2(\Omega)}\norm{v(t)}^{1/2}_{L^2(\Omega)}
	\norm{\nablah \varphi^{(n)}(t)}^{1/2}_{L^2(\Omega)}\norm{\varphi^{(n)}(t)}^{1/2}_{L^2(\Omega)}\\
	&\qquad\qquad\qquad \cdot 
	\left(\|\partial_z \nablah v(t)\|^{1/2}_{L^{2}(\Omega)}\|\nablah v(t)\|^{1/2}_{L^{2}(\Omega)}+\|\nablah v(t)\|_{L^{2}(\Omega)}\right) \dt\\
	&\qquad\leq c\norm{v}^{1/2}_{L^{\infty}((0,T),L^2(\Omega))} \norm{\varphi^{(n)}}^{1/2}_{L^{\infty}((0,T),L^2(\Omega))}\\
	&\qquad\qquad\qquad \cdot \int_0^T \norm{\nablah v(t)}^{1/2}_{L^2(\Omega)}
	\norm{\nablah \varphi^{(n)}(t)}^{1/2}_{L^2(\Omega)} \|\nablah v(t)\|_{H^1((-h,h),L^{2}(G))} \dt\\
	&\qquad\leq c\norm{v}^{1/2}_{L^{\infty}((0,T),L^2(\Omega))}\norm{\nablah v}^{1/2}_{L^2((0,T),L^2(\Omega))} \|\nablah v\|_{L^2((0,T),H^1((-h,h),L^{2}(G)))}\\
	&\qquad\qquad\qquad \cdot \norm{\nablah \varphi^{(n)}}^{1/2}_{L^2((0,T),L^2(\Omega))} 	\norm{\varphi^{(n)}}^{1/2}_{L^{\infty}((0,T),L^2(\Omega))},
	\end{align*}
	and analogously we get by using Lemma~\ref{lemma:lowreg}~a) with $f=\partial_z v (t)$, $g=\varphi^{(n)}(t)$ and $h= w(t)$
	\begin{align*}
	& \int_0^T \left|\left< w(t)\partial_z  v(t), \varphi^{(n)}(t) \right>_{\Omega} \right|\dt \\
	&\qquad\leq c\norm{\partial_z v}^{1/2}_{L^{\infty}((0,T),L^2(\Omega))}\norm{\nablah \partial_z v}^{1/2}_{L^2((0,T),L^2(\Omega))} \|w\|_{L^2((0,T),H^1((-h,h),L^{2}(G)))}\\
	&\qquad\qquad\qquad \cdot \norm{\nablah \varphi^{(n)}}^{1/2}_{L^2((0,T),L^2(\Omega))} 	\norm{\varphi^{(n)}}^{1/2}_{L^{\infty}((0,T),L^2(\Omega))}.
	\end{align*}
	Considering these estimates for $\varphi^{(n)}- \varphi^{(m)}$ gives the convergence
	\begin{align*}
	\left< v(t)\cdot \nablah v(t), \varphi^{(n)}(t) \right>_{\Omega}
	\to  \left<v(t) \cdot \nablah v(t), \varphi(t) \right>_{\Omega}
	\end{align*}
	and
	\begin{align*}
	\left< w(t)\partial_z  v(t), \varphi^{(n)}(t) \right>_{\Omega}\to \left< w(t)\partial_z  v(t), \varphi(t) \right>_{\Omega}
	\end{align*} 
	in $L^1((0,T))$.
\end{proof}

Moreover, $z$-weak solutions preserve certain $L^q$-regularity vertically reflecting the transport-like behavior in this direction. 
\begin{proposition}[$L^q$-norm of $v_z$ remains bounded for $\partial_z v_0\in L^q$.]\label{prop:lq}
Let $v$ be a $z$-weak solution to \eqref{eq:primeqhorvisc}, \eqref{eq:bc} on $(0,T)$, $T>0$, with initial condition $v_0\in H^1((-h,h),L^2(\Omega)^2)$ with $\divh \overline{v_0}=0$ and force $f\equiv 0$.   
If in addition $\partial_z v_0\in L^q(\Omega)^2$ for $q> 2$, then $v_z\in L^{\infty}((0,T),L^q(\Omega))^2$.
\end{proposition}

\begin{proof}
We multiply the equation for $\partial_z v$,
\begin{align*}
 \partial_z v_t -\Deltah \partial_z v +\partial_zv\cdot\nablah v+v\cdot\nablah
 \partial_z v-\divh v \partial_z v + w\partial_{zz}v=0,
\end{align*}
by $|\partial_z v|^{q-2}\partial_z v$ and get
\begin{align*}
&\tfrac{1}{q} \tddt\|\partial_z v\|_{L^q}^q+(q-1)\||\partial_zv|^{(q-2)/2}
\nablah \partial_z v\|_{L^2}^2\\
& = -\left< \partial_zv\cdot\nablah v+ v\cdot\nablah \partial_z v-\divh v \partial_z v
+w\partial_{zz}v,|\partial_z v|^{q-2}\partial_z v\right>_{\Omega}\\
& = -\left< \partial_zv\cdot\nablah v,|\partial_z v|^{q-2}\partial_z v\right>_{\Omega}
+\left< \divh v \partial_z v,|\partial_z v|^{q-2}\partial_z v\right>_{\Omega}.
\end{align*}
Using Lemma \ref{lemma:lowreg}~a) with $f=|\partial_z v|^{q/2}=g$ and 
$h=\nablah v$, we obtain
\begin{align*}
|\left< \partial_zv\cdot\nabla v,|\partial_z v|^{q-2}\partial_z v\right>_{\Omega}|
& \leq c \norm{\nablah |\partial_z v|^{q/2}}_{L^2}\norm{|\partial_z v|^{q/2}}_{L^2} \|\nablah v\|_{H^1_zL^{2}_{xy}}\\
& \leq c \norm{|\partial_zv|^{(q-2)/2}\nablah \partial_z v}_{L^2}
\norm{\partial_z v}^{q/2}_{L^q}\| \nablah v\|_{H^1_zL^{2}_{xy}}\\
& \leq \frac12 \norm{|\partial_zv|^{(q-2)/2}\nablah \partial_z v}_{L^2}
+c \|\nablah v\|^2_{H^1_zL^{2}_{xy}} \norm{\partial_z v}^{q}_{L^q}
\end{align*}
and
\begin{align*}
|\left< \divh v \partial_z v,|\partial_z v|^{q-2}\partial_z v\right>_{\Omega}|
 \leq \frac12 \norm{|\partial_zv|^{(q-2)/2}\nablah \partial_z v}_{L^2}
 +c \|\nablah v\|^2_{H^1_zL^{2}_{xy}} \norm{\partial_z v}^{q}_{L^q}.
\end{align*}
Thus,
\begin{align*}
&\tfrac{1}{q} \tddt \|\partial_z v\|_{L^q}^q+(q-2)\||\partial_zv|^{(q-2)/2}
\nablah \partial_z v\|_{L^2}^2\leq c \|\nablah v\|^2_{H^1_zL^{2}_{xy}} \norm{\partial_z v}^{q}_{L^q}
\end{align*}
and this implies
\begin{align*}
\|\partial_z v(t)\|_{L^q}^q \leq \|\partial_z v_0\|_{L^q}^q \; e^{\|\nablah v\|^2_{L^2((0,T),H^1_zL^{2}_{xy})}},
\end{align*}
so $\partial_z v\in L^{\infty}((0,T),L^q(\Omega))$.
\end{proof}

\section{Local solutions with force}\label{sec:loc}
Theorem~\ref{thm:IWP_loc} $(a)$ and $(b)$ correspond to Proposition~\ref{prop:zweaksolutions} and~\ref{prop:locstrong}, respectively.
	
\subsection{Local $z$-weak solutions}\label{subsec:zweak}
We work in the spaces
\begin{align*}
 H:=\{v\in L^2(\Omega)^2| \partial_z v\in L^2(\Omega)^2\}\quad \hbox{and} \quad H_{\overline{\sigma}}:= H \cap \lso(\Omega)
\end{align*} 
equipped with the scalar product $\left< u,v\right>_H:=\left< u,v
\right>_{\Omega}+\left< \partial_z u,\partial_z v\right>_{\Omega}$ and
\begin{align*}
 V:=\{v\in H| \partial_x v, \partial_y v\in H\}\cap H^1_{0,l}(\Omega)^2
\quad \hbox{and} \quad V_{\overline{\sigma}}:=V\cap \lso(\Omega) 
\end{align*}
with the scalar product $\left< u,v\right>_V:=\left< u,v\right>_{H}+\left< 
\nablah u,\nablah v\right>_{H}$. Note that $H=H^1((-h,h),L^2(G))^2$ and $V=H^1((-h,h),H^1_0(G))^2$. By $V'$ we denote the space 
\begin{align*}
 V'=H^1((-h,h),H^{-1}(G))^2.
\end{align*}
We also denote the dual pairing in $V\times V'$  by $\left< \cdot ,\cdot\right>_H$ 
to keep the notation simple.

\begin{proposition}[Existence and uniqueness of local $z$-weak solutions]\label{prop:zweaksolutions}
Let $v_0\in H_{\overline{\sigma}}$ and $f\in L^2((0,T),V')$, then there exists a $T'\in (0,T]$ such that there is a unique $z$-weak solution of the primitive equations on $(0,T')$ with $v\in C^0([0,T'],L^2(\Omega))$. If $||v_0||_H$ and $||f||_{L^2((0,T),V')}$ are sufficiently small, then $T'=T$.
\end{proposition}

\begin{proof}
We subdivide the proof of the existence and uniqueness into several steps.

\textit{Step 1 (Galerkin approximation).}
To define a suitable basis for a Galerkin scheme, one can take advantage of \eqref{eq:lso_split2}. To this end,
let $(\varphi_m)_m \subset C^{\infty}(\overline{G})^2\cap H^1_0(G)^2$ be an
orthonormal basis of eigenfunctions to the eigenvalues $(\mu_m)_m$ of the Dirichlet Laplacian in $L^2(G)^2$, and $(\tilde{\varphi}_m)_m \subset 
C^{\infty}(\overline{G})^2\cap H^1_0(G)^2\cap \ls(G)$ an orthonormal basis of eigenfunctions to the eigenvalues $(\tilde{\mu}_m)_m$ of the Stokes operator in $\ls (G)$. 
Moreover, $\cos\left(\frac{k\pi}{2h}(z+h)\right)$ for $k\in \N_0$ defines a basis of eigenfunctions to the Neumann Laplacian on $L^2(-h,h)$ and by the first representation theorem even a basis on $H^1(-h,h)$. 

Hence, we define for $m\in \N, k\in\N_0$ the functions $\Phi_{m,k}\in C^{\infty}(\overline{\Omega})$ by
\begin{align} \label{eq:Phi}
 \Phi_{m,k}(x,y,z)&:= \frac{1}{h}\varphi_m(x,y)\cos\left(\frac{k\pi}{2h}(z+h)\right)\quad \hbox{for} \quad k>0,\\ \label{eq:Phi0}
 \Phi_{m,0}(x,y,z)&:= \frac{1}{2h}\tilde{\varphi}_m(x,y).
\end{align}
Then $\operatorname{span} \{\Phi_{m,k} | m\in\N,k\in\N_0\}$ is dense 
in $H_{\overline{\sigma}}$, in particular $\divh\overline{\Phi}_{m,k}=0$, because for $k>0$ we have 
already  $\overline{\Phi}_{m,k}=0$ and $\divh\overline{\Phi}_{m,0}=
\divh \tilde{\varphi}_m=0$. 
We set 
\begin{align}\label{eq:HnPn}
H_{\overline{\sigma},n}:=\operatorname{span}\{\Phi_{m,k} | m,k\leq n \} \quad \hbox{and} \quad P_n: H_{\overline{\sigma}}\to H_{\overline{\sigma},n}
\end{align} 
to be the orthogonal projection onto it.
We project the primitive equations onto the finite dimensional space $H_{\overline{\sigma},n}$ and 
we are looking for a solution $$v_n=(v_{1,n},v_{2,n})\colon[0,T] \to H_{\overline{\sigma},n}$$ of the 
system of ordinary differential equations
\begin{multline}\label{eq:galerkin2}
\tddt \left< v_n,\Phi_{m,k}\right>_{H} + \left<v_n\cdot\nablah v_n, 
 \Phi_{m,k}\right>_{H}  +  \left< w_n \partial_z v_n,\Phi_{m,k}
 \right>_{H}-  \left< \Deltah v_n,\Phi_{m,k}\right>_{H} \\ = \left<f,\Phi_{m,k}\right>_{H}
\end{multline}
for $m,k \leq n$, where we already used $\left<\nablah p,\Phi_{m,k}
\right>_{\Omega}=0$, with initial condition $v_n(0)=P_n v_0$ and $w_n(t,x,y,z)
=-\int_{-h}^z\divh v_n(t,x,y,s)\d\! s$. Note that the properties of $\Phi_{m,k}$ 
imply $\divh \overline{v}_n=0$ and thus we have $w_n(z=\pm h)=0$. Now, we can represent 
\begin{align*}
v_{n}(t)=\sum_{m,k\leq n} g_{n}^{(mk)}(t)\Phi_{m,k} \quad \hbox{for some} \quad g_{n}^{(mk)}:[0,T]
\to\R.
\end{align*}
The existence of a solution $v_n\in H^1((0,T),H_{\overline{\sigma},n})$ follows from classical ODE theory.

\textit{Step 2 ($L^2$-estimate).}
Next, we prove an estimate for $v_n$ in $\lso(\Omega)$. Integrating by parts, for $k>0$ it holds for any function 
$g\in V'$ that
\begin{align*}
 \left<g,\Phi_{m,k}\right>_H 
 =\left<g,\Phi_{m,k}\right>_{\Omega}+ \left< g, -\partial_{zz} \Phi_{m,k}\right>_{\Omega}
 =\left(1+ \frac{k^2\pi^2}{4h^2}\right)\left<g,\Phi_{m,k}\right>_{\Omega},
\end{align*}
and for $k=0$ the corresponding equality holds, because $\partial_z \Phi_{m,0}=0$. Thus \eqref{eq:galerkin2} yields
\begin{multline}\label{eq:galerkin21}
 \left< \partial_t v_n,\Phi_{m,k}\right>_{\Omega} + \left<v_n\cdot\nablah v_n, 
 \Phi_{m,k}\right>_{\Omega}  +  \left< w_n \partial_z v_n,\Phi_{m,k}
 \right>_{\Omega}-  \left< \Deltah v_n,\Phi_{m,k}\right>_{\Omega} \\ = \left<f,\Phi_{m,k}\right>_{\Omega}.
\end{multline}
We multiply this equation by $g^{(mk)}_{n}$ and sum over $m,k\leq n$. 
It then follows that
\begin{align*}
 \tddt \|v_n\|_{L^2(\Omega)}^2+2 \|\nablah v_n\|^2_{L^2(\Omega)}
 & =2\left<f,v_n\right>_{\Omega}-2\left<v_n\cdot\nablah v_n, 
 v_n\right>_{\Omega}-2\left< w_n \partial_z v_n,v_n \right>_{\Omega}\\
 & =2\left<f,v_n\right>_{\Omega}\\
 & \leq \frac{1}{\varepsilon}\|f\|_{L^2((-h,h),H^{-1}(G))}^2+\varepsilon\|v_n\|^2_{L^2((-h,h),H^{1}(G))}, 
\end{align*}
for $\varepsilon>0$, where we used $w_n(\pm h)=0$ when integrating by parts in the vertical direction to show that $\left<v_n\cdot\nablah v_n, 
v_n\right>_{\Omega}+\left< w_n \partial_z v_n,v_n \right>_{\Omega}=0$. 
Thus, for $\varepsilon $ small enough, using the Poincar\'e inequality of Lemma~\ref{lem:poincare}, there exists a constant $C>0$ independent of $v_n$ such that 
\begin{multline} \label{eq:L2_Energy}
 \|v_n\|_{L^{\infty}((0,T),L^2(\Omega))}+ \|\nablah v_n\|_{L^2((0,T),L^2(\Omega))} \\ 
 \leq C\|v_0\|_{L^2(\Omega)}+C\|f\|_{L^2((0,T),L^2((-h,h),H^{-1}(G)))}.
\end{multline}

\textit{Step 3 ($H$-estimate).}
To derive now an estimate for $v_n$ in $H$ we multiply \eqref{eq:galerkin2} 
by $g^{(mk)}_{n}$ and sum over $m,k\leq n$. This gives
\begin{align*}
& \tfrac12  \tddt \|v_n\|^2_{L^2(\Omega)}+\tfrac12  \tddt\|\partial_z v_n\|^2_{L^2(\Omega)}
+\|\nablah v_n\|^2_{L^2(\Omega)}+\|\partial_z \nablah v_n\|^2_{L^2(\Omega)}\\
& \qquad =\left< f,v_n \right>_{\Omega}+\left< \partial_z f,\partial_z v_n \right>_{\Omega}
-\left<\partial_z(v_n\cdot\nablah v_n),\partial_z v_n\right>_{\Omega}
 -\left<\partial_z(w_n\partial_z v_n),\partial_z v_n\right>_{\Omega} \\
& \qquad =\left< f,v_n \right>_{\Omega}+\left< \partial_z f,\partial_z v_n \right>_{\Omega}
-\left<(\partial_z v_n)\cdot\nablah v_n,\partial_z v_n\right>_{\Omega}
 +\left<\divh v_n\partial_z v_n,\partial_z v_n\right>_{\Omega}.
\end{align*}
Here one has used the cancellation property for the non-linear term with respect to $\langle\cdot,\cdot\rangle_\Omega$. Note that with respect to $\langle\cdot,\cdot\rangle_H$ one does not have such  cancellation in general.
Now, Lemma~\ref{lemma:lowreg}~a) yields with $f=g=\partial_z v_n$ and $h=\nablah v_n$ that
\begin{align*}
& |\left<(\partial_z v_n)\cdot\nablah v_n,\partial_z v_n\right>_{\Omega}| \\
& \quad\leq c \|\partial_z \nablah v_n\|^{1/2}_{L^{2}(\Omega)}\|\nablah v_n\|^{1/2}_{L^{2}(\Omega)}
\|\partial_z \nablah v_n\|_{L^2(\Omega)}\|\partial_z v_n\|_{L^2(\Omega)}\\
& \qquad +c\|\partial_z \nablah v_n\|_{L^2(\Omega)}\|\nablah v_n\|_{L^{2}(\Omega)}
\|\partial_z v_n\|_{L^2(\Omega)}\\
& \quad\leq \varepsilon \|\partial_z \nablah v_n\|^2_{L^2(\Omega)}
+c\varepsilon^{-1}\|\nablah v_n\|^2_{L^{2}(\Omega)} \left(\|\partial_z v_n\|_{L^2(\Omega)}^2
+\|\partial_z v_n\|_{L^2(\Omega)}^4\right)
\end{align*}
for any $\varepsilon>0$. An analogous estimate holds for the term $|\left<\divh v_n\partial_z v_n,
\partial_z v_n\right>_{\Omega}|$. So, we obtain by choosing $\varepsilon=\frac{1}{4}$
\begin{align}\label{eq:H_energydt}
\tddt\|v_n\|^2_{H}+ \|\nablah v_n\|^2_{H} \leq \|f\|^2_{V'}+\tfrac{1}{2}\|v_n\|_{V}^2
+c \|\nablah v_n\|^2_{L^{2}(\Omega)} \left(\|v_n\|_{H}^2+\|v_n\|_{H}^4\right)
\end{align}
and integrating with respect to time gives for $t>0$
\begin{multline*}
\|v_n(t)\|^2_{H}+\frac12 \int_0^t\|\nablah v_n(s)\|^2_{H} \d\! s \leq\|v_0\|^2_{H}
+\|f\|^2_{L^2((0,t),V')}+\frac{1}{2}\int_0^t\|v_n(s)\|^2_{H}\d\! s\\
+c \int_0^t\|\nablah v_n(s)\|^2_{L^{2}(\Omega)} \left(\|v_n(s)\|_{H}^2+\|v_n(s)\|_{H}^4\right)\d\! s.
\end{multline*}
Using Poincar\'{e}'s inequality, see Lemma~\ref{lem:poincare}, leads to a situation where a non-linear version of Gr\"onwall's Lemma -- recapped here in Lemma~\ref{lem:nonlinGron} -- is applicable, i.e.,  
\begin{align}\label{eq:H_energy}
\|v_n(t)\|^2_{H}  \leq\|v_0\|^2_{H}
+\|f\|^2_{L^2((0,t),V')}
+c \int_0^t\Psi(s) \omega(\|v_n(s)\|_{H}^2) \d\! s
\end{align}
where $\omega(s)=1+s+s^2$, $\Psi(s):=\|\nablah v_n(s)\|^2_{L^{2}(\Omega)}$, and 
\begin{align*}
\Phi(u) = \int_{0}^u \frac{\d\! s}{\omega(s)} = \frac{2}{\sqrt{3}} \arctan \left(\frac{1+2u}{\sqrt{3}} \right) + c_0,
\end{align*}
where $c_0$ is such that $\Phi(0)=0$.
Due to the boundedness of $\int_0^t\|\nablah v_n(s)\|^2_{L^{2}(\Omega)}\ds$ by the $L^2$-estimate in Step 2, Lemma~\ref{lem:nonlinGron} implies 
\begin{align*}
\|v_n(t)\|^2_{H}  \leq \Phi^{-1}\left( \Phi(\|v_0\|^2_{H}
+\|f\|^2_{L^2((0,t),V')}) +  \int_0^t\|\nablah v_n(s)\|^2_{L^{2}(\Omega)}\d\! s  \right). 
\end{align*}
This is well-defined provided that $ \Phi(\|v_0\|^2_{H}
+\|f\|^2_{L^2((0,t),V')}) +  \int_0^{t}\|\nablah v_n(s)\|^2_{L^{2}(\Omega)}\d\! s$ lies in the range of $\Phi$ which can be assured for small times $0<t\leq T'$, where $T'\in (0,T]$ is sufficiently small.
Using the monotonicity of $\Phi$ which implys the one of $\Phi^{-1}$ and the energy inequality \eqref{eq:L2_Energy} one even obtains that  
\begin{align*}
\|v_n(t)\|^2_{H}  &\leq \Phi^{-1}\left( \Phi(\|v_0\|^2_{H}
+\|f\|^2_{L^2((0,t),V')}) +  \|v_0\|^2_{L^2(\Omega)}
+c\|f\|^2_{L^2((0,t),V')}  \right). 
\end{align*}
Hence for small times $T'\in (0,T]$ or for data $\|v_0\|^2_{L^2(\Omega)}+\|f\|^2_{L^2((0,t),V')}$ being sufficiently small when $T'=T$, one obtains using the continuity of $\Phi^{-1}$ that for some $c>0$
\begin{align}\label{eq:Henergy_small}
\|v_n(t)\|^2_{H}+\int_0^t\|\nablah v_n(s)\|^2_{H} \d\! s \leq c(\|v_0\|^2_{H}+\|f\|^2_{L^2((0,t),V')}), \quad t\in (0,T'].
\end{align}

\textit{Step 4 (Convergence).}
On this interval $[0,T']$ we can deduce 
the weak convergence of a subsequence of $(v_n)_n$ in $L^2((0,T),H_{\overline{\sigma}})$ (which we do not rename) to some 
limit $v\in L^2((0,T),H_{\overline{\sigma}})$. The energy estimate \eqref{eq:Henergy_small} for the sequence gives that $\|v\|^2_{L^{\infty}((0,T),H)}$ 
and $\| \nablah v\|_{L^2((0,T),H)}$ remain bounded, and hence $v$ is in the regularity class of weak and $z$-weak solutions. 

To show that the limit is in fact a weak solution, one takes into account that especially the full gradient of $v_n$ is uniformly bounded in $L^2((0,T),L^2(\Omega ))$ and 
from the compact embedding in the Rellich-Kondrachov theorem the strong convergence of $(v_n)_n$ in $L^2((0,T),L^2(\Omega))$ follows. This 
implies that in $L^1((0,T),L^1(\Omega))$ 
\begin{align*}
v_n\cdot \nablah v_n\rightharpoonup v\cdot \nablah v,\quad w_n v_n \rightharpoonup w v \quad \hbox{and}\quad \partial_z w_n v_n \rightharpoonup \partial_z w v.
\end{align*}
Let now $\varphi_n\in C^{1}([0,T],H_n\cap V)$ be 
of the form $\varphi_n(t)=\sum_{m,k \leq n} h_{m,k}(t)\Phi_{m,k}$ with
$h_{m,k}\in C^{1}([0,T],\R)$. From $w(z=\pm h)$ we conclude
\begin{align*}
\left< w_n \partial_z v_n,  \varphi_n \right>_{\Omega}=-\left< \partial_z w_n v_n,
\varphi_n \right>_{\Omega}-\left< w_n v_n,   \partial_z \varphi_n \right>_{\Omega},
\end{align*}
and using \eqref{eq:galerkin21} we get
\begin{multline*}
\int_0^T -\left<  v_n , \partial_t \varphi_n \right>_{\Omega}  
+\left<  v_n\cdot\nablah  v_n, \varphi_n \right>_{\Omega} 
-\left< \partial_z w_n v_n,  \varphi_n \right>_{\Omega}
-\left< w_n v_n,   \partial_z \varphi_n \right>_{\Omega}\\
+\left< \nablah  v_n, \nablah  \varphi_n \right>_{\Omega}\d\! t 
+\left< v_n(T), \varphi_n(T) \right>_{\Omega}
-\left< v_n(0), \varphi_n(0) \right>_{\Omega}
= \int_0^T  \left< f, \varphi_n \right>_{\Omega}  \d\! t
\end{multline*}
and passing to the limit $n\to \infty$ gives
\begin{multline*}
 -\left<  v , \partial_t \varphi \right>_{\Omega,T}  
+\left<  v\cdot\nablah  v, \varphi \right>_{\Omega,T} 
-\left< \partial_z w v,  \varphi \right>_{\Omega,T}
-\left< w v,   \partial_z \varphi \right>_{\Omega,T}
 \\ +\left< \nablah  v, \nablah  \varphi \right>_{\Omega,T}
 +\left< v(T), \varphi(T) \right>_{\Omega}
 -\left< v(0), \varphi(0) \right>_{\Omega}
=\left< f, \varphi \right>_{\Omega,T} .
\end{multline*}
Showing that $v(0)=v_0$ and $v(T)$ are well-defined follows from the next step which only uses the Galerkin approximation and the convergence.

\textit{Step 5 (Continuity in time).} 
For $l > n$ we extend $v_n\in H^1((0,T),H_{\overline{\sigma},n})$ to $H^1((0,T),H_{\overline{\sigma},l})\subset C^0([0,T],L^2(\Omega))$ by setting $g_n^{(mk)}=0$ if $m\geq l$ or $k\geq l$. For $n<l\in\N$ we set $u:=v_n-v_l$ and $w_u:=w_n-w_l$, where $v_n,v_l$ are elements of the convergent subsequence. From \eqref{eq:galerkin21} it follows that
\begin{multline*}
 \left< \partial_t u,\Phi_{m,k}\right>_{H} + \left<v_n\cdot\nablah v_n-v_l\cdot\nablah v_l, 
 \Phi_{m,k}\right>_{H}  +  \left< w_n \partial_z v_n-w_l \partial_z v_l,\Phi_{m,k}
 \right>_{H}\\
 -  \left< \Deltah u,\Phi_{m,k}\right>_{H}  = \left<f,\Phi_{m,k}\right>_{H}.
\end{multline*}
 Hence,
\begin{align*}
&\tfrac12 \tddt \|u\|_{L^2(\Omega)}^2+\|\nablah u\|_{L^2(\Omega)}^2\\
&\qquad\quad = \left<v_{l}\cdot\nablah  v_{l}- v_{n}\cdot\nablah  v_{n} , u\right>_{\Omega}
+\left< w_{l} \partial_z v_{l} - w_{n} \partial_z v_{n} , u\right>_{\Omega}+\left<f, u\right>_{\Omega}\\
&\qquad\quad =\left< (w_{l}-w_{n}) \partial_z v_{n}, u\right>_{\Omega} 
+\left< w_{l} \partial_z(v_{l}- v_{n}) , u\right>_{\Omega}\\
&\qquad\qquad+\left<(v_{l}-v_{n})\cdot\nablah  v_{n} , u\right>_{\Omega}
+\left<v_{l}\cdot\nablah (v_{l}- v_{n}), u\right>_{\Omega}+\left<f, u\right>_{\Omega}\\
&\qquad\quad = \left< w_{u} \partial_z v_{n}, u\right>_{\Omega}
+\left<u\cdot\nablah  v_{n} , u\right>_{\Omega}+\left<f, u\right>_{\Omega}.
\end{align*}
As in the proof of Proposition \ref{lem:testfunctions} it follows from Lemma \ref{lemma:lowreg}~a) 
with $f=u$, $g=\partial_z v_{n}$ and $h=w_{u}$ that
\begin{align*}
|\left< w_{u} \partial_z v_{n}, u\right>_{\Omega}|
& \leq c \norm{\nablah u}^{1/2}_{L^2(\Omega)}\norm{u}^{1/2}_{L^2(\Omega)}
\norm{\nablah \partial_z v_{n}}^{1/2}_{L^2(\Omega)}
\norm{\partial_z v_{n}}^{1/2}_{L^2(\Omega)}\\
& \qquad\qquad\qquad\qquad\qquad \cdot \left(\|\partial_z w_{u}\|^{1/2}_{L^{2}(\Omega)}
\|w_{u}\|^{1/2}_{L^{2}(\Omega)}+\|w_{u}\|_{L^{2}(\Omega)}\right)\\
& \leq c \norm{\nablah u}^{1/2}_{L^2(\Omega)}\norm{u}^{1/2}_{L^2(\Omega)}
\norm{\nablah \partial_z v_{n}}^{1/2}_{L^2(\Omega)}
\norm{\partial_z v_{n}}^{1/2}_{L^2(\Omega)}\|\nablah u\|_{L^{2}(\Omega)}\\
&\leq \frac14  \norm{\nablah u}^{2}_{L^2(\Omega)}
+c\norm{\nablah v_{n}}^{2}_{H}\norm{v_{n}}^{2}_{H}\norm{u}^{2}_{L^2(\Omega)}
\end{align*}
and similarly from Lemma \ref{lemma:lowreg}~a) with $f=g=u$ and $h=\nabla_Hv_{n}$ that
\begin{align*}
|\left<u\cdot\nablah  v_{n} , u\right>_{\Omega}|
& \leq c \norm{\nablah u}_{L^2(\Omega)}\norm{u}_{L^2(\Omega)}\\
& \qquad\qquad\qquad \cdot
\left(\|\partial_z \nablah  v_{n} \|^{1/2}_{L^{2}(\Omega)}
\|\nablah v_{n} \|^{1/2}_{L^{2}(\Omega)}+\|\nablah  v_{n} \|_{L^{2}(\Omega)}\right)\\
& \leq \frac14\norm{\nablah u}^{2}_{L^2(\Omega)}
+c\norm{\nablah v_{n}}^{2}_{H}\norm{u}^2_{L^2(\Omega)}.
\end{align*}
Altogether, we have
\begin{multline*}
\|u(t)\|_{L^2(\Omega)}^2+\int_0^t\|\nablah u(s)\|_{L^2(\Omega)}^2\d\! s 
\leq \int_0^t\left<f(s), u(s)\right>_{\Omega}\ds \\
+ c\int_0^t\left(1+\norm{\nablah v_{n}(s)}^{2}_{H}+\norm{\nablah v_{n}(s)}^{2}_{H} \norm{v_{n}(s)}^{2}_{H}\right) \|u(s)\|^2_{L^2(\Omega)}\d\! s,
\end{multline*}
and because of 
\begin{multline*}
\int_0^t 1+\norm{\nablah v_{n}(s)}^{2}_{H}+\norm{\nablah v_{n}(s)}^{2}_{H} \norm{v_{n}(s)}^{2}_{H}\d\! s\\
\leq C(T+\norm{\nablah v_{n}}^{2}_{L^{2}((0,T),H)}+\norm{v_{n}}^{2}_{L^{\infty}((0,T),H)}\norm{\nablah v_{n}}^{2}_{L^{2}((0,T),H)}) < \infty
\end{multline*}
uniformly in $n$, we get
\begin{align*}
\|u(t)\|_{L^2(\Omega)}^2\leq c\left(\|u(0)\|_{L^2(\Omega)}^2+ \int_0^T\left|\left<f(s), u(s)\right>_{\Omega}\right|\ds \right).
\end{align*}
The weak-$\ast$-convergence of $v_n \to v$ in $L^2((0,T),V)$ yields $\int_0^T\left|\left<f(s), u(s)\right>_{\Omega}\right|\ds\to 0$ ($n,l\to\infty$) and with $\|u(0)\|_{L^2(\Omega)}\to 0$ ($n,l\to\infty$) follows 
\begin{align*}
\sup_{t\in[0,T]}\|u(t)\|_{L^2(\Omega)}^2\leq c\left(\|u(0)\|_{L^2(\Omega)}^2+ \int_0^T\left|\left<f(s), u(s)\right>_{\Omega}\right|\ds \right) \to 0
\end{align*}
for $n,l\to \infty$. Thus $v\in C^0([0,T],L^2(\Omega))$.

\textit{Step 6 (Uniqueness).} In this step the estimates on the non-linear terms are similar to the above. Let $v_1,v_2$ be two $z$-weak solutions to the same initial datum $v_0\in H_{\overline{\sigma}}$.
Set $u:=v_{1}-v_{2}$ 
and $w_{u}:=w_{1}-w_{2}$ for $w_{1}=w(v_{1})$ and $w_{2}=w(v_{2})$. Then we have $u(t=0)= 0$ in $L^2(\Omega )$ and 
\begin{align*}
\partial_t u -\Deltah  u
+ \PP (v_{1}\cdot\nablah  v_{1}- v_{2}\cdot\nablah  v_{2}
+ w_{1} \partial_z v_{1} -  w_{2} \partial_z v_{2})
=0.
\end{align*}
We have $u \in L^{\infty}((0,T),\lso(\Omega )) \cap L^2((0,T),L^2((-h,h),H^1_0(G))^2)$. Hence, 
Proposition~\ref{lem:testfunctions} yields that we can 
test the above equation with $u$, and similar to Step~5
\begin{align*}
\tfrac12 \tddt \|u\|_{L^2(\Omega)}^2+\|\nablah u\|_{L^2(\Omega)}^2
= \left< w_{u} \partial_z v_{1}, u\right>_{\Omega}
+\left<u\cdot\nablah  v_{1} , u\right>_{\Omega},
\end{align*}
where
\begin{align*}
|\left< w_{u} \partial_z v_{1}, u\right>_{\Omega}|
&\leq \frac14  \norm{\nablah u}^{2}_{L^2(\Omega)}
+\norm{\nablah v_{1}}^{2}_{H}\norm{v_{1}}^{2}_{H}\norm{u}^{2}_{L^2(\Omega)} \hbox{ and } \\
|\left<u\cdot\nablah  v_{1} , u\right>_{\Omega}|
& \leq \frac14\norm{\nablah u}^{2}_{L^2(\Omega)}
+c\norm{\nablah v_{1}}^{2}_{H}\norm{u}^2_{L^2(\Omega)}.
\end{align*}
It follows
\begin{multline}\label{eq:u_estimate}
\|u(t)\|_{L^2(\Omega)}^2+\int_0^t\|\nablah u(s)\|_{L^2(\Omega)}^2\d\! s\\
\leq  \int_0^t\left(  \norm{\nablah v_{1}(s)}^{2}_{H} \norm{v_{1}(s)}^{2}_{H}+ \norm{\nablah v_{1}(s)}^{2}_{H} \norm{v_{1}(s)}^{2}_{H}\right) \|u(s)\|^2_{L^2(\Omega)}\d\! s,
\end{multline}
and we can apply Gr\"onwall's inequality to obtain $u=0$.
\end{proof}

\subsection{Local strong solutions}\label{subsec:locstrong}




\begin{proposition}[Local strong well-posedness]\label{prop:locstrong}
Let $v_0\in H_{0,l}^1(\Omega)^2\cap \lso(\Omega)$ and $f\in L^2((0,T),H)$. 
Then the local $z$-weak solution on $(0,T')$ for $T'\in (0,T]$ given by Proposition~\ref{prop:zweaksolutions} has the additional regularity
\begin{align*}
 \nablah v \in L^{\infty}((0,T'),L^2(\Omega)), \quad 
 \Deltah v, \partial_t v \in L^2((0,T'),L^2(\Omega)).
\end{align*}
\end{proposition}
\begin{proof}
Recall that 
we have obtained in the proof of Proposition~\ref{prop:zweaksolutions}  a solution $v_n=(v_{1n},v_{2n})\in H^1((0,T),H_{n,\overline{\sigma}})$ of the 
system \eqref{eq:galerkin2} of ordinary differential equations
with $v_n(0)=P_n v_0$. These satisfy the $L^2$-estimate  \eqref{eq:L2_Energy} on $[0,T]$ and the $H$-estimate \eqref{eq:Henergy_small} on some small time interval $[0,T']$ for $T'\in (0,T]$, where $T=T'$ if the data are sufficiently small.

%

\textit{Step 1 ($L^2_t$-$L^2_x$-estimate on $\Deltah v$ and   $L^\infty_t$-$L^2_x$-estimate on $\nablah v$).}
Here we use the higher regularity of the initial data and the fact, that the functions $\Phi_{m,k}$ defined in \eqref{eq:Phi} and \eqref{eq:Phi0}
are a  basis of eigenfunctions to certain operators. 
Multiplication of \eqref{eq:galerkin21} first with the eigenvalue $\mu_m$ of $\varphi_{m}$ or 
respectively $\tilde{\mu}_m$ of $\tilde{\varphi}_{m}$ and second with $g_{n}^{(mk)}$, and then summing over both $m$ and 
$k$ gives
\begin{multline*}
 \left< \partial_t v_n,-\Deltah v_n \right>_{\Omega} + \left<v_n\cdot\nablah v_n, 
 -\Deltah v_n\right>_{\Omega}  +  \left< w_n \partial_z v_n,-\Deltah v_n
 \right>_{\Omega}+  \left<\Deltah v_n,\Deltah v_n\right>_{\Omega} \\ 
 = \left<f,-\Deltah v_n  \right>_{\Omega},
\end{multline*}
and it follows that
\begin{multline*}
 \tddt \|\nablah v_n \|_{L^2(\Omega)}^2 + \|\Deltah v_n
 \|_{L^2(\Omega)}^2 \\ 
 \leq 2\left<v_n\cdot\nablah v_n,\Deltah v_n\right>_{\Omega}+2\left< 
 w_n \partial_z v_n,\Deltah v_n\right>_{\Omega}+\|f\|_{L^2(\Omega)}^2.
\end{multline*}
Using Lemma~\ref{lemma:lowreg}~b) with $f=\Deltah v_n$, $g=v_n$ and $h=\nabla_Hv_n$ we get
\begin{align*}
& |\left< v_n\cdot\nablah v_n,\Deltah v_n \right>_{\Omega}| \\
&\qquad \leq c \norm{\Deltah v_n}_{L^2(\Omega)}  \norm{\nablah^2 v_n}_{L^2(\Omega)}^{1/2} 
\norm{\nablah v_n}_{L^2(\Omega)}^{1/2}\\
&\qquad \qquad \cdot \norm{v_n}_{H}^{1/2}
(\norm{v_n}_{L^2(\Omega)} +\norm{\nablah v_n}_{L^2(\Omega)})^{1/2}\\
&\qquad \leq c\norm{\Deltah v_n}_{L^2\Omega)}^{3/2}
(\norm{\nablah v_n}_{L^2(\Omega)}^{1/2}\norm{v_n}_{H}
+\norm{\nablah v_n}_{L^2(\Omega)}\norm{v_n}_{H}^{1/2})\\
&\qquad \leq \frac18 \norm{\Deltah v_n}_{L^2(\Omega)}^{2}+c(\norm{v_n}^4_{H}
+\norm{\nablah v_n}^2_{L^2(\Omega)}\norm{v_n}_{H}^{2})\norm{\nablah v_n}^2_{L^2(\Omega)},
\end{align*}
and by Lemma~\ref{lemma:lowreg}~b) with $f=\Deltah v_n$, $g=w_n$ and $h=\partial_zv_n$
\begin{align*}
&|\left< w_n \partial_z v_n,\Deltah v_n \right>_{\Omega}| \\
& \qquad\leq c \norm{\Deltah v_n}_{L^2(\Omega)}
\norm{\nablah \partial_z v_n}_{L^2(\Omega)}^{1/2} 
\norm{\partial_z v_n}_{L^2(\Omega)}^{1/2}\\
& \qquad\qquad \cdot (\norm{w_n }_{L^2(\Omega)} 
+\norm{\partial_z w_n }_{L^2(\Omega)})^{1/2}
(\norm{w_n}_{L^2(\Omega)} +\norm{\nablah w_n }_{L^2(\Omega)})^{1/2}\\
& \qquad \leq c \norm{\Deltah v_n}_{L^2(\Omega)} 
\norm{\nablah \partial_z v_n}_{L^2(\Omega)}^{1/2} 
\norm{\partial_z v_n}_{L^2(\Omega)}^{1/2}
\norm{\nablah v_n}_{L^2(\Omega)}^{1/2}\\
& \qquad\qquad \cdot (\norm{\nablah v_n}_{L^2(\Omega)} 
+\norm{\nablah^2 v_n}_{L^2(\Omega)})^{1/2}\\
& \qquad\leq c  \norm{v_n}_{V}^{1/2}\norm{v_n}_{H}^{1/2}\\
& \qquad\qquad \cdot \left(\norm{\Deltah v_n}_{L^2(\Omega)}
\norm{\nablah v_n}_{L^2(\Omega)}+\norm{\nablah v_n}_{L^2(\Omega)}^{1/2}
\norm{\Deltah v_n}_{L^2(\Omega)}^{3/2}\right)\\
& \qquad \leq \frac18 \norm{\Deltah v_n}_{L^2(\Omega)}^{2}
+c\norm{\nablah v_n}_{L^2(\Omega)}^2
\left(\norm{v_n}_{V}\norm{v_n}_{H}
+\norm{v_n}^2_{V}\norm{v_n}^2_{H}\right).
\end{align*}
So, for some $c>0$ we have the estimate 
\begin{multline}\label{eq:energyH1}
\tddt\|\nablah v_n\|_{L^2(\Omega)}^2+\tfrac12\|\Deltah v_n\|_{L^2(\Omega)}^2
 \leq \|f\|_{L^2(\Omega)}^2\\
  + c\left(\norm{v_n}_{V}\norm{v_n}_{H}
+\norm{v_n}^2_{V}\norm{v_n}^2_{H}+\norm{v_n}^4_{H}+\norm{\nablah v_n}^2_{L^2}
\norm{v_n}_{H}^{2}\right)\norm{\nablah v_n}_{L^2}^{2}.
\end{multline}
Our previous estimate \eqref{eq:Henergy_small} shows, that the pre-factor of $\norm{\nablah v_n}_{L^2}^{2}$ 
on the right hand side has a bounded time integral for small times or small data and thus Gr\"onwall's lemma implies 
that 
\begin{multline}\label{eq:H1energy_small}
\|\nablah v_n\|_{L^{\infty}((0,T),L^2(\Omega))}^2
+\|\Deltah v_n\|_{L^2((0,T),L^2(\Omega))}^2\\
\leq c\left(\|\nablah v_0\|_{L^2(\Omega)}^2+ \|f\|_{L^2((0,T),L^2(\Omega))}^2\right).
\end{multline}

\textit{Step 2 ($L^2_t$-$L^2_x$-estimate on $\partial_t v$).}
To obtain a better regularity in time we multiply \eqref{eq:galerkin21} 
with $\partial_t g_{n}^{(mk)}$ and sum over $m$ and $k$. It follows that
\begin{multline*}
 \left< \partial_t v_n, \partial_t v_n \right>_{\Omega} + \left<v_n\cdot\nablah v_n, 
  \partial_t v_n\right>_{\Omega}  +  \left< w_n \partial_z v_n, \partial_t v_n
 \right>_{\Omega}-  \left< \Deltah v_n, \partial_t v_n\right>_{\Omega} \\ 
 = \left<f, \partial_t v_n \right>_{\Omega}.
\end{multline*}
Similar to the above we get by Lemma~\ref{lemma:lowreg}~b) with $f=\partial_tv_n$, $g=v_n$ and $h=\nabla_Hv_n$ that
\begin{align*}
 |\left< v_n\cdot\nablah v_n,\partial_t v_n \right>_{\Omega}| 
& \leq c \norm{\partial_t v_n}_{L^2(\Omega)}  \norm{\nablah^2 v_n}_{L^2(\Omega)}^{1/2} 
\norm{\nablah v_n}_{L^2(\Omega)}^{1/2}\\
& \qquad \cdot \norm{v_n}_{H}^{1/2}
(\norm{v_n}_{L^2(\Omega)} +\norm{\nablah v_n}_{L^2(\Omega)})^{1/2}\\
& \leq c\norm{\partial_t v_n}_{L^2(\Omega)}
\norm{\Deltah v_n}_{L^2(\Omega)}^{1/2} \norm{ v_n}_{H^1(\Omega)}^{3/2}\\
& \leq \frac18\norm{\partial_t v_n}^2_{L^2(\Omega)}+c
\norm{\Deltah v_n}_{L^2(\Omega)}\norm{v_n}_{H^1(\Omega)}^{3},
\end{align*}
where we used that due to the ellipticity of the Laplacian the second derivatives $\nablah^2 v_n$ can be bounded in $L^2(G)$ 
and hence also in $L^2(\Omega)$ by the horizontal 
Laplacian $\Deltah v_n$.
By Lemma~\ref{lemma:lowreg}~b) with $f=\partial_tv_n$, $g=w_n$ and $h=\partial_zv_n$ we obtain that
\begin{align*}
|\left< w_n \partial_z v_n,\partial_t v_n  \right>_{\Omega}| 
& \leq c \norm{\partial_t v_n}_{L^2(\Omega)} 
\norm{\nablah \partial_z v_n}_{L^2(\Omega)}^{1/2} 
\norm{\partial_z v_n}_{L^2(\Omega)}^{1/2}
\norm{\nablah v_n}_{L^2(\Omega)}^{1/2}\\
& \qquad \cdot (\norm{\nablah v_n}_{L^2(\Omega)} 
+\norm{\Deltah v_n}_{L^2(\Omega)})^{1/2}\\
& \leq c \norm{\partial_t v_n}_{L^2(\Omega)}
\norm{v_n}_{V}^{1/2} \norm{v_n}_{H^1(\Omega)}\\
& \qquad \cdot(\norm{v_n}_{H^1(\Omega)}+\norm{\Deltah v_n}_{L^2(\Omega)})^{1/2}\\
& \leq \frac18\norm{\partial_t v_n}^2_{L^2(\Omega)}\\
 &\qquad\quad +c\norm{v_n}_{V} \norm{v_n}^2_{H^1(\Omega)}
 (\norm{v_n}_{H^1(\Omega)}+\norm{\Deltah v_n}_{L^2(\Omega)})
\end{align*}
which leads to
\begin{multline*}
\tddt\|\nablah v_n\|_{L^2(\Omega)}^2+ \|\partial_t v_n\|^2_{L^2(\Omega)}
\leq 2\|f\|_{L^2(\Omega)}^2\\
+c\left(\norm{v_n}_{V} \norm{v_n}^2_{H^1(\Omega)}
(\norm{v_n}_{H^1(\Omega)}+\norm{\Deltah v_n}_{L^2(\Omega)})
+\norm{\Deltah v_n}_{L^2(\Omega)}\norm{v_n}_{H^1(\Omega)}^{3}\right).
\end{multline*}
The time integral of the right hand side is bounded by Step 1 and \eqref{eq:Henergy_small} assuming smallness of time or data, 
there is a $c>0$ with
\begin{align*}
 \|\partial_t v_n\|^2_{L^2((0,T),L^2(\Omega))}
\leq c(\|\nablah v_0\|_{L^2(\Omega)}^2+
\|f\|_{L^2((0,T),L^2(\Omega))}^2).
\end{align*}
Now we pass to the limit as in the proof of the existence of a $z$-weak solution 
and we get, that  for a subsequence (which we do not rename) also $\partial_t v_n$ and $\Deltah v_n$ converge weakly in 
$L^2((0,T),L^2(\Omega))$ and that the limit of $\nablah v_n$ is in 
$L^{\infty}((0,T),L^2(\Omega))$.


\end{proof}
%

\section{Time-periodic solutions for small forces}\label{sec:per}
The methods used to prove global existence and uniqueness results for the initial value problem for small data can be adapted for the construction of time-periodic solutions. This will be done first for $z$-weak and then for strong solutions. 
%
%

\begin{proposition}[Existence and uniqueness of time-periodic solutions] \label{prop:periodiczweak} For $\varepsilon_b>0$  there exists $\varepsilon_a>0$ such that 
	\begin{enumerate}
		\item[(a)] if $f\in L^2((0,T),V')$ with
		$\|f \|_{L^2((0,T),V')}<\varepsilon_a$, then there exists a  $T$-periodic $z$-weak solution with $\sup_{0 <s<T}\|v_n(s)\|^2_{H} +\|\nablah v_n(s)\|^2_{L^2((0,T),H)} <\varepsilon_b$;
		\item[(b)] if $f\in L^2((0,T),H)$
		with $\|f \|_{L^2((0,T),H)}<\varepsilon_a$, then there exists a  $T$-periodic strong solution with $\sup_{0 <s<T}\|v_n(s)\|^2_{H^1}+ \|\Deltah v_n\|_{L^2((0,T),L^2(\Omega))}^2<\varepsilon_b$.
	\end{enumerate}
\end{proposition}

\begin{proof}[Proof of Proposition~\ref{prop:periodiczweak}]
		Consider as in the proof of Theorem~\ref{prop:zweaksolutions} the finite dimensional spaces $H_n$. Adapting the strategy of \cite{Galdi2017}, we consider 
		the Poincar\'e map
		\begin{align*}
		H_n \rightarrow H_n, \quad  v_{n}(0) \mapsto v_n(T),
		\end{align*}
		where $v_n$ is the solution to \eqref{eq:galerkin2}.
		
	\textit{Step 1 (Existence of a $z$-weak solution).}	
	Note that for $f\in L^2((0,T),V')$ the differential energy inequality \eqref{eq:H_energydt} can be modified 
	using the Poincar\'e inequality from Lemma~\ref{lem:poincare} 
	to become for some $c>0$
	\begin{align*}
	\tddt\|v_n\|^2_{H}+ c\|v_n\|_{H}^2 \leq c\|f\|_{V'}^2
	+c \|\nablah v_n\|^2_{L^{2}(\Omega)} \left(\|v_n\|_{H}^2+\|v_n\|_{H}^4\right),
	\end{align*}
	and multiplying by $e^{sc}$ and integrating with respect to $t$ this becomes
	\begin{multline*}
	e^{tc}\|v_n(t)\|^2_{H}\leq \|v_n(0)\|^2_{H} + c\int_0^te^{sc}\|f(s)\|^2_{V'} \d\! s \\
	+c \int_0^te^{sc}\|\nablah v_n(s)\|^2_{L^{2}(\Omega)} \left(\|v_n(s)\|_{H}^2+\|v_n(s)\|_{H}^4\right) \d\! s.
	\end{multline*}
	Assuming that $\|v_n(0)\|_{H}$ and $\|f\|_{L^2((0,T),V')}$ are sufficiently small, one has from \eqref{eq:Henergy_small} for $t=T$ and \eqref{eq:L2_Energy} that
	\begin{align}\label{eq:perH}
	e^{Tc}\|v_n(T)\|_{H}^2\leq \|v_n(0)\|_{H}^2 
	+c e^{Tc} (\|v_n(0)\|^4_{H} +\|f(s)\|^2_{L^2((0,T),V')}).
	\end{align}
	Assume now that $\|v_n(0)\|^2_{H}\leq R$ for $R\in (0,1)$ being small enough to satisfy the smallness condition and $(1-e^{-Tc})/c>R$, moreover let $\|f(s)\|_{L^2((0,T),L^2(\Omega))}$ be sufficiently small for \eqref{eq:Henergy_small} to hold and 
	\begin{align*}
	\|f(s)\|^2_{L^2((0,T),V')} \leq \frac{1}{ce^{cT}}(e^{cT}-1)R - R^2.
	\end{align*}
	Then
	\begin{align*}
	e^{Tc}\|v_n(T)\|_{H}\leq R 
	+ce^{Tc}R^2 +  (e^{cT}-1)R-e^{Tc}R^2 = Re^{Tc}.
	\end{align*}
	Hence, for this $R>0$, and $B_{R,n}:=\{v_n\in H_{\overline{\sigma},n}\colon \|v_n\|_{H}\leq R  \}$ and given $f$, the map 
	\begin{align*}
	B_{R,n} \rightarrow B_{R,n}, \quad v(0) \mapsto v(T), 
	\end{align*}
	where $v_n$ is the solution to \eqref{eq:galerkin2} is a continuous self-mapping. By Brouwer's fixed point theorem, for any $n\in \N_0$, there is a fixed point, i.e., $v_{n,0}\in H_{n,\overline{\sigma}}$ with $v_{n,0}=v_n(0)=v_n(T)$. Since the $v_{n,0}$ are uniformly bounded in $H$ by $R$ there is a convergent subsequence in $L^2(\Omega)$, the limit of which is in $H$. Following the proof of Proposition~\ref{prop:zweaksolutions}, the approximate solutions $v_n$ converge to a $z$-weak solution with $v(0)=v(T)$.  
	
	By \eqref{eq:Henergy_small} 
	\begin{align*}
	\|v_n(s)\|^2_{L^\infty((0,T),H)} +\|\nablah v_n(s)\|^2_{L^2((0,T),H)} &\leq c(\|v_0\|^2_{H}+\|f\|^2_{L^2((0,T),V')})  \\ 
	&\leq c(R^2 + \varepsilon_a^2)
	\end{align*}
	which for $R$ and $\varepsilon_a$ sufficiently small is smaller $\varepsilon_b$.
	
	\textit{Step 2 (Existence of a strong solution).}	
	Using Lemma~\ref{lem:poincare}, one can modify \eqref{eq:energyH1} to become after multiplying $e^{sc}$ and integrating with respect to $t$
	\begin{multline*}
	e^{tc}\|\nablah v_n(t)\|_{L^2(\Omega)}^2
	\leq \|\nablah v_n(0)\|_{L^2(\Omega)}^2 + c\int_0^te^{sc} \|f(s)\|_{L^2(\Omega)}^2 \d\! s\\
	+ c\int_0^te^{sc}\bigg(\norm{v_n(s)}_{V}\norm{v_n(s)}_{H}
	+\norm{v_n(s)}^2_{V}\norm{v_n(s)}^2_{H}\\ +\norm{v_n(s)}^4_{H}+\norm{\nablah v_n(s)}^2_{L^2}
	\norm{v_n(s)}_{H}^{2}\bigg)\norm{\nablah v_n(s)}_{L^2}^{2} \d\! s.
	\end{multline*}
	Assuming that $\|v_0\|_{{H^1}(\Omega)}$ and $\|f\|_{L^2((0,T),H)}$ are sufficiently small, one can combine this with the previously obtained estimate \eqref{eq:perH} to obtain	
		\begin{align*}
		e^{Tc}\|v_n(T)\|_{H^1(\Omega)}\leq \|v_n(0)\|_{H^1(\Omega)} 
		+c e^{Tc} (\|v_n(0)\|^2_{H^1(\Omega)} +\|f(s)\|_{L^2((0,T),L^2(\Omega))}).
		\end{align*}
	Proceeding now analogously to the above, one proves the existence of a small $T$-periodic solution with $v(0)=v(T)\in H^1(\Omega)^2\cap \lso(\Omega)$ sufficiently small. By Proposition~\ref{prop:locstrong} this is a strong solution. 
	
	The norm estimate follows as above, but now by combining \eqref{eq:Henergy_small} with \eqref{eq:H1energy_small}.
\end{proof}

\begin{proof}[Proof of Theorem~\ref{thm:periodic}]
	The existence of $T$-periodic solutions in Theorem~\ref{thm:periodic} (a) and (b) follows directly from Proposition~\ref{prop:periodiczweak} (a) for forces  $\|f \|_{L^2((0,T),V')}<\varepsilon_a$ and (b) for forces  $\|f \|_{L^2((0,T),H)}<\varepsilon_a$, respectively. 
	
	Now, to prove the uniqueness let $v_1$ be the $T$-periodic $z$-weak solution constructed in Proposition~\ref{prop:periodiczweak} satisfying the required smallness assumption, and let
	and $v_2$ be another $T$-periodic $z$-weak solutions for the same $f\in L^2((0,T),V')$.
	As in Step 6 in the proof of Theorem~\ref{prop:zweaksolutions}, 
	one considers $u:=v_1-v_2$, 
	and then it holds that
	\begin{align*}
	\tddt \|u\|_{L^2(\Omega)}^2+\|\nablah u\|_{L^2(\Omega)}^2\leq c(\norm{\nablah v_{1}}^{2}_{H} + \norm{\nablah v_{1}}^{2}_{H}\norm{v_{1}}^{2}_{H})\norm{u}^{2}_{L^2(\Omega)},
	\end{align*}
	compare \eqref{eq:u_estimate}, and hence by Poincar\'e's inequality, cf. Lemma~\ref{lem:poincare}, 
	\begin{align*}
	\tddt \|u\|_{L^2(\Omega)}^2\leq c(\norm{\nablah v_1}^{2}_{H} + \norm{\nablah v_1}^{2}_{H}\norm{v_1}^{2}_{H}-C)\norm{u}^{2}_{L^2(\Omega)}.
	\end{align*}
	By the differential form of Gr\"onwall's inequality
	\begin{multline*}
	\|u(0)\|_{L^2(\Omega)}^2=\|u(T)\|_{L^2(\Omega)}^2\\
	\leq \|u(0)\|_{L^2(\Omega)}^2 e^{c\int_0^T\norm{\nablah v_1(s)}^{2}_{H} + \norm{\nablah v_1(s)}^{2}_{H}\norm{v_1(s)}^{2}_{H} \d\! s-TC}.
	\end{multline*}
	and for $c\int_0^T\norm{\nablah v_1(s)}^{2}_{H} + \norm{\nablah v_1(s)}^{2}_{H}\norm{v_1(s)}^{2}_{H} \d\! s<TC$, one has a factor smaller than one, and hence $\|u(0)\|_{L^2(\Omega)}^2=0$ which implies by the uniqueness for the initial value problem uniqueness of the $T$-periodic solutions.
	 This condition holds provided that $\varepsilon_b$ is so small 
	 that $C(\varepsilon_b + \varepsilon_b^2) < CT$, and one chooses $\varepsilon=\varepsilon_a$ with the corresponding $\varepsilon_a$.
\end{proof}


\section{Global strong solutions}\label{sec:glob}
The main idea is to establish first a global \textit{a priori} bound for some smooth data, and second to use both the partial parabolic smoothing in the horizontal directions and the conservation of regularity in vertical direction to show that some $z$-weak solutions reach this setting for $t>0$.

\subsection{Global \textit{a priori} bound}\label{subsec:apriori}
For the global bound on strong solutions (cf. Proposition~\ref{prop_unif_bound}), we prove a differential inequality of the form \[
f^\prime \le \Vert v\Vert _{L^\infty} ^2f,\]
where $f$ contains certain Sobolev-norms of the solution, via performing the first order estimates (cf. Proposition~\ref{prop_vert_fo_est} and \ref{prop_hor_fo_est}). As in  \cite{CaoLiTiti2017}, to control the $\Vert v\Vert _\infty$-coefficient, we use the logarithmic Sobolev inequality (cf. Proposition~\ref{prop_log_sob}) and show that the $L^q$-norm of the solutions grow asymptotically at most as $\sqrt{q}$ (cf. Proposition~\ref{prop_lq_est}). Then the classical Gr\"onwall lemma gives the desired bound. This implies global existence via a standard contradiction argument.
To prove the logarithmic Sobolev inequality in our setting we need the following extension result. 

\begin{lemma}\label{lem:extension}
	Let $q_H,q_z\in (1,\infty )$ and $f\in L^{q_H}((-h,h),H^{1,q_H}_0(G))$, such that $\partial _zf\in L^{q_z}(\Omega )$.
	Then there exists an extension $\tilde{f}:\mathbb{R}^3\rightarrow \mathbb{R}$, such that \[
	\Vert f\Vert _{L^\infty}\le \Vert \tilde{f}\Vert _{L^\infty (\mathbb{R}^3)},\quad \Vert \tilde{f}\Vert _{L^q(\mathbb{R}^3)}\le C\Vert f\Vert _{L^q},\quad \Vert \partial_i\tilde{f}\Vert _{L^{q_i}(\mathbb{R}^3)}\le C\Vert \partial _if\Vert _{L^{q_i}},
	\]
	for all $q\in (1,\infty )$ and $(q_i,\partial_i)\in \{(q_H,\nabla_H ),(q_z,\partial _z)\}$.\end{lemma}
\begin{proof}
	The idea is, vertically, to reflect $f$ on $z=h$, extend it periodically to a function $\hat{f}:G\times \mathbb{R}\rightarrow \mathbb{R}$ and cut it off thereafter. Horizontally, we simply extend it by zero. More explicitly
	\[\tilde{f}(x,y,z)=\left \{\begin{array}{ll}\phi (z)\hat{f}(x,y,z)&(x,y)\in G,\\
	0& \text{else,}
	\end{array} \right .\]
	where
	\[\hat{f}(x,y,z)=\left \{\begin{array}{ll}
	f(x,y,z+4kh)&\exists k\in \mathbb{Z}:z+4kh\in[-h,h],\\
	f(x,y,2h-(z+4kh))&\exists k\in \mathbb{Z}: z+4kh\in [h,3h],
	\end{array} \right .\]
	for some $\phi \in C_0^\infty (\mathbb{R})$, such that $\phi \equiv 1$ on $(-h,h)$, and $0\le \phi \le 1$ on $\mathbb{R}$.
\end{proof}
\begin{proposition}[Logarithmic Sobolev inequality]\label{prop_log_sob}Let
 $p=(p_1,p_2,p_3)\in(1,\infty )^3$ with $p_H=p_1=p_2$ and $\sum ^3_{i=1}p_i^{-1}<1$. Then for any $F\in L^{p_H}((-h,h),H^{1,p_H}_0(G))$ such that $\partial _zF\in L^{p_3}(\Omega )$, we have
	\[\Vert F\Vert _{L^\infty }\le C_{p,\lambda ,\Omega}\max \left \{ 1,\sup_{q\ge 2}\frac{\Vert F\Vert _{L^q}}{q^\lambda}\right \}\log ^\lambda \left ( e+\sum ^3_{i=1}(\Vert F\Vert _{L^{p_i}}+\Vert \partial _iF\Vert _{L^{p_i}})\right ),\]
	for any $\lambda >0$ when all the norms are finite.
\end{proposition}
\begin{proof}
	By the previous Lemma~\ref{lem:extension}, there exists an extension $\tilde{F}$ of $F$ to the whole space such that\[
	\Vert F\Vert _{L^\infty }\le \Vert \tilde{F}\Vert _{L^\infty},\quad \Vert \tilde{F}\Vert _{L^q(\mathbb{R}^3)}\le C\Vert F\Vert _{L^q },\quad \Vert \partial_i\tilde{F}\Vert _{L^{p_i}(\mathbb{R}^3)}\le C\Vert \partial _iF\Vert _{p_i}.
	\]
	Thus, it follows from the logarithmic Sobolev inequality on the whole space (cf. e.g. \cite[Lemma 5.1]{CaoLiTiti2017}) that\begin{align*}
	&\, \Vert F\Vert _{L^\infty } \le \Vert \tilde{F}\Vert _{L^\infty (\mathbb{R}^3)}\\
	\le &\, C_{p,\lambda}\max\left \{1,\sup_{q\ge 2}\frac{\Vert \tilde{F}\Vert _{L^q(\mathbb{R}^3)}}{q^\lambda}\right \}\log ^\lambda \left (e+\sum^3_{i=1}(\Vert \tilde{F}\Vert _{L^{p_i}(\mathbb{R}^3)}+\Vert \partial_i\tilde{F}\Vert _{L^{p_i}(\mathbb{R}^3)} )\right )\\
	\le &\, C_{p,\lambda ,\Omega}\max\left \{1,\sup_{q\ge 2}\frac{\Vert {F}\Vert _{L^q }}{q^\lambda }\right \}\log ^\lambda \left (e+\sum^3_{i=1}(\Vert {F}\Vert _{L^{p_i}}+\Vert \partial_i{F}\Vert _{L^{p_i}} )\right ),
	\end{align*}finishing the proof.
\end{proof}
Assume from now on that $v$ is a strong  and sufficiently smooth solution to the primitive equations on $(0,T)$ with initial condition $v_0$ and $f\equiv 0$.
On the way of showing, that the $L^q$-norm of the solution grows asymptotically at most of order $\mathcal{O}(\sqrt{q})$ we need to prove that the term $v\cdot\nabla_Hv$ lies in $L^2((0,T),L^2(\Omega ))$. To this end we need estimates on $\norm{\nablah \overline{v}}_{L^2 }$. As initial step, recall that by testing with $v$ one obtains that the energy equality for strong solutions of the primitive equations holds for almost all $t\in (0,T)$
\begin{align}\label{eq:ei}
\norm{v(t)}_{L^2(\Omega)}^2 + 2\int_0^t \norm{\nablah v(s)}_ {L^2(\Omega)}^2 \d\! s =\norm{v(0)}^2_{L^2(\Omega)}.
\end{align}

\begin{lemma}[$L^2$- and $L^4$-estimates]\label{le1_l4_est}
It holds
	\begin{multline*}
	\sup _{0\le t\le {T}}(\Vert \nablah \overline{v}\Vert^2_{L^2 }+\Vert \tilde{v}\Vert^4_{L^4 }) + \int_0^{{T}} \norm{\Deltah \overline{v}}^2_{L^2 } \d\! t +\int_0^{{T}}\Vert \vert \tilde{v}\vert \nabla _H\tilde{v}\Vert _{L^2 }^2\d\! t \\ \le K({T})(1+\Vert \tilde{v}_0\Vert_{L^4 }^4 + \Vert \nablah \overline{v}_0\Vert _{L^2 }^2),
	\end{multline*}
	where $K:[0,\infty )\rightarrow \mathbb{R}$ is a continuously increasing function determined by $h$, $\Vert \tilde{v}_0\Vert_{L^4 }^4$ and $\Vert \nablah \overline{v}_0\Vert^2_{L^2 }$.
\end{lemma}
\begin{proof}
	We shall only give a sketch of the proof. Recall, the momentum equation of the problem splits into an equation for $\overline{v}$ on $G$
	\begin{equation} \label{eq:bar}
	\begin{array}{rll}
	\partial_t \bar v - \Deltah \bar v + \nabla_Hp &=-\bar v\cdot\nabla_H\bar v - \frac{1}{2h} \int_{-h}^{h} (\tilde v\cdot\nabla_H\tilde v + (\mathrm{div}_H\, \tilde v)\,\tilde v)\,\d\! z,  &  \\
	\mathrm{div}_H\,\bar v &= 0, &  \\
	\bar v(0) &= \bar v_0, 
	\end{array}
	\end{equation}
	and an equation for $\tilde{v}$ on $\Omega$
	\begin{equation}\label{eq:tilde}
	\begin{array}{rll}
	\partial_t\tilde v - \Deltah\tilde v    = & - \tilde v\cdot\nabla_H\tilde v - w(\tilde v)\tilde{v}_z - \bar v\cdot\nabla_H\tilde v
	- \tilde v\cdot\nabla_H\bar v  \\
		& + \frac{1}{2h} \int_{-h}^h (\tilde v\cdot\nabla_H\tilde v + (\mathrm{div}_H\,\tilde v)\,\tilde v)\,\d\! z,  &  \\
	\tilde v(0)  = & \tilde{v}_0.
	\end{array}
	\end{equation}
	Similarly to \cite[Section 6, Step 1]{HieberKashiwabara2015}, one first multiplies \eqref{eq:bar} by $\PP_G \Deltah \overline{v}$, and then integrates over $G$. When integrating by parts the pressure gradient vanishes, and applying a compensation argument yields
	\begin{multline*}
	\tfrac{1}{2}\tddt\norm{\nablah\overline{v}}^2_{L^2(G)} +  \norm{\PP_G \Deltah \overline{v}}^2_{L^2(G)}  
	=  
	\\ \int_{\Omega} \left(-\bar v\cdot\nabla_H\bar v - \frac{1}{2h} \int_{-h}^h (\tilde v\cdot\nabla_H\tilde v + (\mathrm{div}_H\, \tilde v)\,\tilde v)\,\d\! z  \right)\PP_G \Deltah \overline{v}\d (x,y,z)
	\\
	\leq 
	C \norm{|\overline{v}|\nablah \overline{v} }^2_{L^2(G)}
	+ 4 \norm{|\tilde{v}|\cdot |\nablah \tilde{v}| }^2_{L^2(\Omega)}
	+ \frac{1}{4}\norm{\PP_G \Deltah \overline{v}}^2_{L^2(G)} 
	\\
	\leq 
	C\norm{\overline{v}}_{L^4(G)}^2\norm{\nablah\overline{v}}_{L^4(G)}^2
	+ 4 \norm{|\tilde{v}|\nablah \tilde{v} }^2_{L^2(\Omega)}
	+ \frac{1}{4}\norm{\PP_G \Deltah \overline{v}}^2_{L^2(G)} \\
	\leq
	C\norm{\overline{v}}^2_{L^2(G)}\norm{\nablah v}^2_{L^2(\Omega)} \norm{\nablah\overline{v}}_{L^2(G)}^2 
	+
	\frac{1}{2}\norm{\PP_G \Deltah\overline{v}}^2_{L^2(G)}
	+ 
	4\norm{|\tilde{v}|\nablah \tilde{v} }^2_{L^2(\Omega)},
	\end{multline*}
	where one uses H\"older's and Young's inequality along with Ladyzhenskaya's inequality and ellipticity of the $2$-D Stokes operator  $\PP_G \Deltah$ with domain contained in $H^2(G)^2$. Note that $C>0$ depends only on $h$. 
	Hence,
	\begin{multline}\label{eq:H1bar}
	\norm{\nablah\overline{v}(t)}^2_{L^2(G)} + \int_0^t \norm{\PP_G \Delta \overline{v}(s)}^2_{L^2(G)}\d\! s \leq  \norm{\nablah\overline{v}(0)}^2_{L^2(G)}  \\ +
	C \int_0^t \norm{\overline{v}}_{L^2(G)}^2\norm{\nablah v}^2_{L^2(\Omega)} \norm{\nablah\overline{v}}_{L^2(G)}^2 \d\! s 
	+8 \int_0^t \norm{|\tilde{v}|\nablah \tilde{v} }^2_{L^2(\Omega)}\d\! s .
	\end{multline}	
Next one follows \cite[Section 6, Step 3]{HieberKashiwabara2015}, i.e., testing \eqref{eq:tilde} with $\tilde{v}\vert \tilde{v}\vert^2$ yields
	\begin{multline*}
	\frac{1}{4}\ddt\Vert \tilde{v}\Vert _{L^4 }^4 +\int _\Omega\vert \tilde{v}\vert ^{2} \left[\vert \nabla _H \tilde{v}\vert ^2+2\big\vert \nabla _H\vert \tilde{v}\vert \big\vert ^2\right]\d (x,y,z)=
	-\int _\Omega  \tilde{v}\cdot \nablah \overline{v} \cdot \tilde{v}\vert \tilde{v}\vert ^2 \d (x,y,z)\\
	+\int _\Omega \frac{1}{2h} \int_{-h}^h (\tilde v\cdot\nabla_H\tilde v + (\mathrm{div}_H\,\tilde v)\,\tilde v)\,\d\! z \tilde{v}\vert \tilde{v}\vert ^2\d (x,y,z).
	\end{multline*}
	Analogously to 	\cite[Section 6, Step 3, estimates on $I_7$ and $I_8$]{HieberKashiwabara2015} one obtains
	\begin{align*}
	\int _\Omega  \tilde{v}\cdot \nablah \overline{v} \cdot \tilde{v}\vert \tilde{v}\vert ^2 \d (x,y,z)\leq 
	C\|\nabla_H\bar v\|_{L^2(G)}^2 \|\tilde v\|_{L^4(\Omega)}^4 + \frac14\big\|\nabla_H|\tilde v|^2 \big\|_{L_2(\Omega)}^2.
	\end{align*}
	and
	\begin{multline*}
	\int _\Omega \frac{1}{2h} \int_{-h}^h (\tilde v\cdot\nabla_H\tilde v + (\mathrm{div}_H\,\tilde v)\,\tilde v)\,\d\! z \tilde{v}\vert \tilde{v}\vert ^2 \d (x,y,z) \\
	\leq 
	C\|\nabla_H  v\|_{L^2(G)}^2 \|\tilde v\|_{L^4(\Omega)}^4 + \frac14\big\|\nabla_H|\tilde v|^2 \big\|_{L_2(\Omega)}^2.
	\end{multline*}
	Hence,
	\begin{multline}\label{eq:L4tilde}
	\Vert \tilde{v}(t)\Vert _{L^4 }^4 + 2\int_0^t \norm{\vert \tilde{v}\vert \vert \nabla _H \tilde{v}\vert}_{L^2}^2+  \norm{\vert \tilde{v}\vert \nabla _H\vert \tilde{v}\vert }_{L^2 }^2 \d\! s \\ \leq
	\Vert \tilde{v}_0\Vert _{L^4 }^4 + C \int_0^t 
	\|\nabla_H  v\|_{L^2(G)}^2 \|\tilde v\|_{L^4(\Omega)}^4\d\! s.
	\end{multline}
	Adding now \eqref{eq:L4tilde} to $\frac{1}{16}$-times \eqref{eq:H1bar} gives with $\norm{\overline{v}(t)}^2_{L^2(G)}\leq \norm{v_0}^2_{L^2(\Omega)}$
	\begin{multline*}
	\Vert \tilde{v}(t)\Vert _{L^4 }^4 + \norm{\nablah\overline{v}}^2_{L^2(G)}
	+\int_0^{t} \norm{\PP \Deltah \overline{v}}^2_{L^2 }\d\! s +\int_0^{t}\Vert \vert \tilde{v}\vert \nabla _H\tilde{v}\Vert _{L^2 }^2 \d\! s
	 \\ \leq\Vert \tilde{v}(0)\Vert _{L^4 }^4 +8 \norm{\nablah\overline{v}(0)}^2_{L^2(G)}  \\ 
	+ C \int_0^t (\norm{v_0}^2_{L^2(\Omega)}+1)
	\norm{\nablah v}^2_{L^2(\Omega)} (\norm{\nablah\overline{v}}^2_{L^2(G)} + \norm{\tilde{v}}^4_{L^4 })\d\! s.
	\end{multline*}
	Using Gr\"onwall's inequality the claim follows with \\ $K(T)=(\Vert \tilde{v}(0))\Vert _{L^4 }^4+8\Vert \nabla _H\overline{v}(0)\Vert _{L^2(G)})K'(T)$, where
	\begin{align*}
	K'({T})=e^{C \max\{1,\norm{v(0)}^2_{L^2(\Omega)}\}\int_0^{T}\norm{\nablah v}^2_{L^2(\Omega)}\d\! s} \leq 
	e^{C \max\{1,\norm{v_0}^2_{L^2(\Omega)}\}\norm{v(0)}^2_{L^2(\Omega )}}. \quad\qedhere 
	\end{align*}
\end{proof}
Now, using the decomposition $v=\overline{v}+\tilde{v}$ one shows the integrability of
\begin{align*}
\Vert \divh \overline{v\otimes v} \Vert _{L^2 }^2
\leq C\Vert \vert \tilde{v}\vert \nabla _H\tilde{v}\Vert _{L^2 }^2
+C\Vert \vert \overline{v}\vert \nabla _H\overline{v}\Vert _{L^2 }^2,
\end{align*}
where the first addend is integrable by Lemma~\ref{le1_l4_est}, and
\begin{align*}
\Vert \vert \overline{v}\vert \nabla _H\overline{v}\Vert _{L^2 }^2
&\leq \norm{\overline{v}}_{\infty}^2   \norm{\nabla _H\overline{v}}_{2}^2
\leq \norm{\overline{v}}_{H^2(G)}^2\norm{\nabla _H\overline{v}}_{2}^2.
\end{align*}
Hence, Lemma~\ref{le1_l4_est} implies the following corollary.
\begin{corollary}[$L^4$-estimate]\label{le_l4_est}
	It holds
	\begin{align}
	\int_0^{{T}}\Vert \divh \overline{v\otimes v} \Vert _{L^2 }^2\d\! t\le K({T})(1+\Vert \tilde{v}_0\Vert_{L^4 }^4 + \Vert \nablah \overline{v}_0\Vert _{L^2 }^2)
	\end{align}
	where $K:[0,\infty )\rightarrow \mathbb{R}$ is a continuously increasing function determined by $h$, $\Vert \tilde{v}_0\Vert_{L^4 }^4$ and $\Vert \nablah \overline{v_0}\Vert^2_{L^2 }$.
\end{corollary}

\begin{proposition}[$L^q$-estimate]\label{prop_lq_est}
	Let $q\ge 4,$ then it holds \[
	\sup _{0\le t\le {T}}\Vert \tilde{v}\Vert _{L^q }\le K_1({T})(1+\Vert \tilde{v}_0\Vert _{L^q })\sqrt{q},
	\]where $K_1:[0,\infty )\rightarrow \mathbb{R}$ is a continuously increasing function determined by $h$, $\Vert v_0\Vert_{L^4 }^4$ and $\Vert \nablah \overline{v_0}\Vert^2_{L^2 }$.
\end{proposition}
For the proof we need the following estimate on the pressure.
\begin{lemma}[Estimate on the pressure for the $2$-D Stokes equations] \label{lem:2DStokes}
	Let $\overline{v},p$ be a solution to the two-dimensional Stokes equations
	\begin{align*}
	\partial_t \overline{v}   - \Deltah \overline{v} + \nablah p =  \overline{f}, \quad \divh\overline{v}=0.
	\end{align*}
	Then for almost every $s\in (0,T)$
	\begin{align*}
	\displaystyle{\Vert \nablah p(s)\Vert_{L^2} \le ||\overline{ f}(s)||_{L^2} + ||\Deltah\overline{v}(s)||_{L^2}}.
	\end{align*}
	provided each term is finite.
\end{lemma}
\begin{proof}
	Applying the complement of the $2$-D Helmholtz projection $\PP_G$, i.e., $\mathds{1}-\PP_G$ to the Stokes equations gives
	\begin{align*}
	\nablah p = (\mathds{1}-\PP_G)(\overline{f}(s)+\Deltah \overline{v}).
	\end{align*}
	Since this is an orthogonal projection in $L^2(G)^2$ one has $\norm{\overline{g}}_{L^2(G)}^2 = \norm{\PP_G \overline{g}}_{L^2(G)}^2+\norm{(\mathds{1}-\PP_G )\overline{g}}_{L^2(G)}^2$ for any $g\in L^2(G)^2$, and hence the claim follows.
\end{proof}

\begin{proof}[Proof of Proposition~\ref{prop_lq_est}]
	We shall only give a sketch of the proof. Recall the momentum equation of the problem\[
	\partial_tv-\Delta_Hv=-\nabla_Hp-\left [(v\cdot\nabla_H)v-\left (\int_{-h}^z\nabla_H\cdot v\d\! z\right )\partial_zv\right ].\]
	Multiplying the above equation by $v\vert v\vert ^{q-2}$, integrating over $\Omega $ yields by integration by parts
	\begin{multline}
	\frac{1}{2}\Vert v\Vert _{L^q }^{q-2}\ddt\Vert v\Vert _{L^q }^2 +\int _\Omega\vert v\vert ^{q-2} \left[\vert \nabla _H v\vert ^2+(q-2)\big\vert \nabla _H\vert v\vert \big\vert ^2\right]\d (x,y,z) \\
	=-\int _\Omega \nabla _Hp\cdot \vert v\vert ^{q-2}v\d (x,y,z)=:I.
	\end{multline}
	Using a series of standard integral inequalities, one can show 
	\begin{multline*}
	I\le C(1+\Vert \nabla _Hp\Vert _{L^2(G)}^2)(1+\Vert v\Vert ^2_{L^2 })\left(q\Vert v\Vert _{L^q }^{q-2}+\Vert v\Vert _{L^q }^{\tfrac{q(q-2)}{q-1}}\right)\\
	+(q-2)\left\Vert \vert v\vert ^{\tfrac{q}{2}-1}\nabla_H\vert v\vert \right\Vert _{L^2 }^2,
	\end{multline*}
	where $C>0$ is independent of $q$. Note that from \eqref{eq:bar} and Lemma~\ref{lem:2DStokes} it follows that
	\begin{align*}
	\Vert \nablah p\Vert _{L^2(G)}^2 \leq \Vert \divh\overline{v\otimes v}\Vert _{L^2(G)}^2 +  \Vert \Deltah \overline{v}\Vert _{L^2(G)}^2.
	\end{align*}
	Combining this with the above we end up with
	\begin{multline*}
	\tddt(q+1+\Vert v\Vert_{L^q }^2)=\tddt\Vert v\Vert_{L^q }^2 \\ 
	\le C(1+\Vert v\Vert ^2_{L^2 })(1+\Vert \divh\overline{v\otimes v}\Vert _{L^2(G)}^2+ \Vert \Deltah \overline{v}\Vert _{L^2(G)}^2)\left(q+1+\Vert v\Vert _{L^q }^2\right).
	\end{multline*}
	The Gr\" onwall inequality implies now
	\begin{align*}
	\sup_{0\le t\le {T}}\Vert v\Vert _{L^q }^2\le K_1^2({T})(q+1+\Vert v_0\Vert _{L^q }^2)\le K_1^2({T})(1+\Vert v_0\Vert _{L^q }^2)q,
	\end{align*}
	where $K_1^2({T})=e^{C\int_0^{T}(1+\Vert v\Vert ^2_{L^2 })(1+\Vert \divh\overline{v\otimes v}\Vert _{L^2(G)}^2)+ \Vert \Deltah \overline{v}\Vert _{L^2(G)}^2\d\! s}$, which is finite due to Corollary~\ref{le_l4_est}.
	
\end{proof}

\begin{proposition}[$L^q$-estimate for $\partial _zv$]\label{prop_vert_fo_est}
	For $q\in [2,\infty )$ it holds
	\[\ddt\Vert \partial _zv\Vert _{L^q }^q+\int _\Omega \vert \partial _zv\vert ^{q-2}\vert \nabla _H\partial _zv\vert ^2\d (x,y,z)\le C_q\left( \Vert v\Vert _{\infty }+1\right)\left( \Vert \partial _zv\Vert _{L^q }^q+1\right).
	\]
\end{proposition}
Note that this bound differs from the local one in Proposition~\ref{prop:lq} by the assumptions on $v$.

\begin{proof} Differentiating the momentum equation with respect to $z$ gives
	\begin{align*}
	\partial _t\partial _zv-\Delta _H\partial _zv
	=&\,(\nabla _H\cdot v)\partial _zv-(\partial _zv\cdot \nabla _H)v+(\int _{-h}^z\nabla_H\cdot v\d\! z)\partial _z^2v -(v\cdot \nabla _H )\partial _zv.
	\end{align*}
	Note that the last two summands of the above equation vanish after multiplication by $\vert \partial _z v\vert ^{q-2}\partial _zv,$ and integration over $\Omega ,$ since 
	\begin{align*}
	&\,\int_\Omega \left [ (\int_{-h}^z\nabla_H\cdot v\d\! z)\partial_z^2v-(v\cdot \nabla_H)\partial_zv\right ]\cdot \vert \partial _z v\vert ^{q-2}\partial _zv\d (x,y,z)\\
	=&\,\frac{1}{q}\int_\Omega \left [(\int_{-h}^z\nabla_H\cdot v\d\! z)\partial_z\vert \partial_zv\vert ^q-(v\cdot \nabla_H)\vert \partial_zv\vert ^q\right ]\d (x,y,z)=0.
	\end{align*}
	Therefore, integrating by parts and  Young's inequality imply\begin{align*}
	&\,\frac{1}{q}\ddt\Vert \partial _zv\Vert ^q_{L^q }+\int _\Omega \vert \partial _zv\vert ^{q-2}\left[ \vert \nabla _H\partial _zv\vert ^2 +(q-2)\left\vert \nabla _H\vert \partial _zv\vert \right\vert ^2\right]\d (x,y,z)\\
	{=}&\,\int _\Omega \left [(\nabla _H \cdot v)\partial _zv-(\partial _zv\cdot \nabla _H)v\right ]\cdot \vert \partial _zv\vert ^{q-2}\partial _zv\d (x,y,z)\\
	{=}&\,-\int_\Omega \left [v\cdot\nabla_H\vert \partial_zv\vert ^q-\nabla_H\cdot\left (\partial_zv\partial_zv^T \vert \partial_zv\vert ^{q-2}\right )\cdot v \right ]\d (x,y,z)\\
	{\le} &\,C_q\int _\Omega \vert v\vert \vert \partial _zv\vert ^{q-1}\vert \nabla _H\partial _zv\vert \d (x,y,z)\\
	{\le} &\,\frac{1}{2}\int _\Omega \vert \partial _zv\vert ^{q-2}\vert \nabla _H\partial _zv\vert ^2\d (x,y,z)+C_q\int _\Omega \vert v\vert ^2\vert \partial _zv\vert ^q\d (x,y,z).
	\end{align*}
	So, subtracting $\frac{1}{2}\int _\Omega \vert \partial _zv\vert ^{q-2}\vert \nabla _H\partial _zv\vert ^2\d (x,y,z)$ from the above inequality and multiplying it by $q$ 
	finishes this proof.\end{proof}

\begin{proposition}[$L^2$-estimate for $\nabla_Hv$]\label{prop_hor_fo_est}It holds for $q>2$ that
	\[
	\tddt\Vert \nabla_Hv\Vert _{L^2 }^2+\Vert\Delta_Hv\Vert _{L^2 }^2\le C\Vert v\Vert _{L^\infty } ^2\Vert\nabla_Hv\Vert _{L^2 }^2+C(\Vert \partial _zv\Vert ^{\frac{4q}{q-2}}_{L^q }+1).\]
\end{proposition}
\begin{proof}
	Multiplying the momentum equation by $-\Delta_Hv$, and integrating over $\Omega $, it follows from integrating by parts that\begin{align*}
	&\frac{1}{2}\ddt\Vert \nabla_Hv\Vert _{L^2 }^2+\int_\Omega \vert \Delta_Hv\vert ^2\d (x,y,z)\\=&\, \int_\Omega \left [(v\cdot\nabla_H)v-\left (\int_{-h}^z\divh v\right )\partial_zv\right ]\cdot\Delta_Hv\d (x,y,z)\\
	\le &\, C\Vert v\Vert _{L^\infty }\Vert \nabla_Hv\Vert _{L^2 }\Vert \Delta_Hv\Vert _{L^2 }+\int _G\int_{-h}^h\vert \nabla_Hv\vert\d\! z\int_{-h}^h\vert\partial_zv\vert\vert \Delta_Hv\vert \d\! z\d (x,y,z).
	\end{align*}
	Using a series of standard integral inequalities, one can show that the later summand of the right hand side can be estimated by suitable terms. More specifically, \[
	\int _G\int_{-h}^h\vert \nabla_Hv\vert\int_{-h}^h\vert\partial_zv\vert\vert \Delta_Hv\vert \le C(1+\Vert\Delta_Hv\Vert ^{\frac{1}{2}-\frac{1}{q}})\Vert \partial_zv\Vert _{L^q }\Vert\Delta_Hv\Vert _{L^2 }.\]
	The above and Young's inequality imply
	\begin{align*}
	&\frac{1}{2}\ddt\Vert \nabla_Hv\Vert _{L^2 }^2+\Vert\Delta_Hv\Vert _{L^2 }^2\\
	&\qquad\qquad \leq C\left [\Vert v\Vert _{L^\infty }\Vert \nabla_Hv\Vert _{L^2 }+(1+\Vert\Delta_Hv\Vert_{L^2 }^{\frac{1}{2}-\frac{1}{q}})\Vert \partial_zv\Vert _{L^q }\right ]\Vert\Delta_Hv\Vert _{L^2 }\\
	&\qquad\qquad \leq \frac{1}{2}\Vert\Delta_Hv\Vert _{L^2 }^2+ C(\Vert v\Vert _{L^\infty } ^2\Vert\nabla_Hv\Vert _{L^2 }^2+\Vert \partial _zv\Vert ^{\frac{4q}{q-2}}_{L^q }+1).
	\end{align*}
	So, subtracting $\frac{1}{2}\Vert\Delta_Hv\Vert _{L^2 }^2$ finishes the proof.
\end{proof}
\begin{proposition}[Uniform \textit{a priori} bound]\label{prop_unif_bound}
	For any finite time ${T}$, we have\[\sup_{t\in [0,{T}]}(\Vert \nabla v\Vert _{L^2 }^2+\Vert \partial _zv\Vert _{2+\eta}^{2+\eta})+\int_0^{{T}}\Vert \nabla_H\nabla v\Vert _{L^2 }^2\d\! t\le C_{\eta ,h, {T}}(\Vert v_0\Vert _{H^1_\eta}),
	\] for an increasing function $C_{\eta ,h, {T}}$ depending only on $\eta $, $h$, ${T}$ and with  $\Vert v_0\Vert _{H^1_\eta}=\Vert v_0\Vert_{H^1}+\Vert\partial_zv_0\Vert _{2+\eta}+\Vert v_0\Vert _{L^\infty }$.
\end{proposition}
Recall that the definition of $\Vert \cdot\Vert _{H^1_\eta}$ is given in Theorem~\ref{thm:IWP_glob}.
\begin{proof}
	Summing up Proposition~\ref{prop_vert_fo_est} and \ref{prop_hor_fo_est}, one can show
	\[ \ddt A(t)+B(t)\le C(1+\Vert v\Vert _{L^\infty }^2)A(t),\]
	where\begin{align*}
	A=A_1+A_1^\lambda+A_2,\quad B=A_1+B_1+B_2,\quad \lambda=4/\eta\\
	A_1=\Vert \partial_zv\Vert_{L^2 }^2+\Vert \partial_zv\Vert_{2+\eta}^{2+\eta}+e,\quad B_1=\Vert \nabla_H\partial_zv\Vert_{2}^{2}\\
	A_2=\Vert\nablah v \Vert _{L^2 }^2+e,\quad B_2=\Vert \Deltah v\Vert_{L^2 }^2+e.
	\end{align*}
	We will show \[
	(1+\Vert v\Vert _{L^\infty }^2)\le C\log B.\] Thus a logarithmic type  Gr\" onwall  inequality (cf. e.g. \cite[Lemma 2.5]{CaoLiTiti2017}) will imply the desired bound. By Proposition~\ref{prop_log_sob}, Proposition~\ref{prop_lq_est} and the Sobolev and horizontal Poincar\' e inequalities, one has\begin{align*}
	\Vert v\Vert_{L^\infty }^2\le &\, C\max\left \{ 1,\sup_{q\ge 2}\frac{\Vert v\Vert _{L^q }^2}{q}\right \} \log (e+\Vert \nabla_Hv\Vert_6+\Vert v\Vert _6+\Vert \partial_zv\Vert _{L^2 }+\Vert v\Vert_{L^2 })\\
	\le &\, C\log (e+\Vert \nabla_Hv\Vert_{H^1}+\Vert v\Vert _{H^1}+\Vert \partial_zv\Vert _{L^2 })\\
	\le &\, C\log (e+\Vert \nabla_Hv\Vert_{2}+\Vert \nabla\nabla_Hv\Vert _{2}+\Vert \partial_zv\Vert _{L^2 })\\
	\le &\, C\log (e+\Vert \Delta_Hv\Vert_{2}+\Vert \nabla_H\partial_zv\Vert _{2}+\Vert \partial_zv\Vert _{L^2 })\le C\log B,
	\end{align*}finishing the proof.
\end{proof}
\begin{proof}[Proof of Theorem \ref{thm:IWP_glob} (a)]
	
	Note first that a local solution can be constructed as in Proposition~\ref{prop:zweaksolutions} and~\ref{prop:locstrong} by a Galerkin scheme using the \textit{a priori} bound from Proposition~\ref{prop_unif_bound}. 
	
	Now, to prove global existence, let $T_{max}$ be the supremum over the existence times of the strong solution, and assume that $T_{max}<\infty$. Choose $$\delta <\max \{T_{max},T^*(C\Vert v_0\Vert _{H^1_\eta})\},$$ where $T^*(K)$ denotes the minimal existence time given by Theorem~\ref{thm:IWP_loc}~(b) of the strong solution of the problem with initial data of norm at most $K$, and $C$ is the constant of the previous Proposition.\\
	Since $\Vert v\Vert_{L^\infty ((0,T_{max}),H^1 )}\le C\Vert v_0\Vert _{H^1_\eta}$ there is a $\hat{T}\in (T_{max}-\delta,T_{max})$, such that $\Vert v(\hat{T})\Vert _{ H^1 }\le C\Vert v_0\Vert _{H^1_\eta}$. By the local existence of strong solutions there exists a strong solution $\hat{v}$ to the problem with initial data $v(\hat{T})$ with an existence time $T^*(\Vert v(\hat{T})\Vert _{H^1})$. Since $T^*:\mathbb{R}^+\rightarrow\mathbb{R}^+$ is monotone decreasing, it holds by the last proposition that \[T^*(\Vert v(\hat{T})\Vert _{H^1})\ge T^*(C\Vert v_0\Vert _{H^1_\eta})>\delta .\]
	Note that $v\vert _{[\hat{T},T_{max}]}=\hat{v}\vert _{[0,T_{max}-\hat{T}]}$. By uniqueness of the solution, $v$ can be extended to $\hat{T}+T^*(\Vert v(\hat{T})\Vert _{H^1})>\hat{T}+\delta >T_{max}$. A contradiction to the maximality of $T_{max}$. So, our assumption of $T_{max}<\infty$ was wrong, finishing the global existence proof.
\end{proof}

\subsection{Extension of $z$-weak solutions to global strong solutions}

\begin{proof}[Proof of Theorem~ \ref{thm:IWP_glob} (b)]
By Theorem~\ref{thm:IWP_loc} (a), there is a $z$-weak solution on $(0,T')$ for some $T'>0$, and
$\nablah v_z \in L^2((0,T'),L^2(\Omega)^{2\times 2})$
 and hence  $v(t)\in H_z^1H^1_{xy}\subset H^1(\Omega)$ for almost every $t\in (0,T')$. This has to be understood in the sense that if one has a smooth approximating sequence $(v_n)\subset C_0^0((0,T'),H_z^1H^1_{xy})$, then there exists a subsequence $(v_{n_k})$ such that $v_{n_k}(t)$ converges for almost every $t\in (0,T')$. 
 
In particular there exists $t_1>0$ with $v(t_1)\in H_{0,l}^1(\Omega)$ and $\divh \overline{v(t_1)}=0$.
 Taking $v(t_1)$ as new initial value one obtains by Theorem~\ref{thm:IWP_loc} (b) a strong solution on $(t_1,T'')$ for some $T''\in (t_1,T')$. This strong solution agrees on $(t_1,T'')$ with the original $z$-weak solution defined on $(0,T')$ since $v_{n_k}(t_1)$ converges also in $L^2$ and since $v\in C^0([0,T'], L^2)$ its $L^2$-limit is in fact $v(t_1)$. Hence by the uniqueness  of $z$-weak solutions, cf. Proposition~\ref{prop:zweaksolutions}, both agree on $(t_1,T'')$. Note that $\norm{v(t_1)}_{H^1}$ is finite, but there is no explicit control on its norm.
 
Recall that the strong solution has the regularity $v\in L^2((t_1,T''),L^2_zH^2_{xy})$ and as $z$-weak solution $v\in L^2((0,T'),H^1_zH^1_{xy})$. 
Consider $\Deltah$ and $\Delta_z=\partial_z^2$ on
\begin{align*}
D(\Deltah)=L^2_z(H_{0,xy}^1\cap L^2_zH_{xy}^2) \quad\hbox{and}\quad D(\Delta_z)=\{ v\in H^2_zL^2_{xy} \colon v_z(z=\pm h)=0 \},
\end{align*}
respectively.
These are commuting self-adjoint operators in Hilbert spaces, and hence one obtains by the mixed derivative theorem, cf. \cite[Corollary III.4.5.10]{PruessSimonett2016}, 
for $\theta\in (0,1)$
\begin{align*}
H^{1}_{z}H^{1}_{0,xy}\cap L^2_zH^2_{xy}\cap L^2_zH^1_{0,xy} &=
D(\Delta_z^{1/2}\circ \Deltah^{1/2}) \cap D(\Deltah)  \\ &\subset D(\Delta_z^{\theta/2}\circ\Deltah^{\theta/2}\circ\Deltah^{1-\theta})\\
&=D(\Delta_z^{\theta/2}\circ\Deltah^{1-\theta/2})\subset H^{\theta}_zH_{xy}^{2-\theta}.
\end{align*}
In particular for $\theta=3/4$ one has $H^{3/4}_zH^{5/4}_{xy}\subset L^\infty_{z}L^\infty_{xy}=L^\infty(\Omega)$.
Moreover, by Proposition~\ref{prop:lq} one has that $v_z\in L^{\infty}((0,T'),L^\eta(\Omega)^2)$  since by assumption here $\partial_z v_0\in L^\eta(\Omega)$. Putting the pieces together, one deduces that 
\begin{align*}
v\in L^2((t_1,T''),H^1(\Omega)^2\cap \{v\in L^2(\Omega)^2\colon v_z\in L^\eta(\Omega)^2\} \cap L^{\infty}(\Omega)^2)
\end{align*}
and therefore
for almost every $t\in (t_1,T'')$ one has
\begin{align*}
v(t)\in H^1_\eta = H^1(\Omega)^2\cap \{v\in L^2(\Omega)^2\colon v_z\in L^\eta(\Omega)^2\} \cap L^{\infty}(\Omega)^2.
\end{align*}
Following the previous arguments one takes now such $t_2\in (t_1,T'')$ as new initial time, and one ends up in the situation of Theorem~\ref{thm:IWP_glob} (a) which gives that $v$ extends to a global strong solution on $(t_2,T)$ for any $T>0$. In fact $0<t_1<t_2<T''$ have been arbitrarily small, and therefore for $\delta\geq t_2$ the statement follows. 

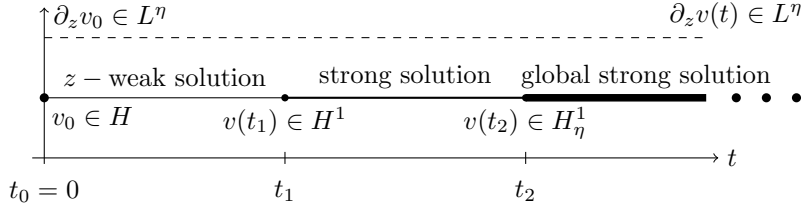
\begin{figure}
\begin{center}

\begin{tikzpicture}[scale=0.8]

    \draw[->] (-0.2,1) -- (11.2,1) node[right] {$t$};
    \draw[->] (0,0.9) -- (0,3.5) node[above] {$ $};

    \draw (0,0.9) node[below=3pt] {$t_0=0$};
    \draw[-] (4,1.1) -- (4,0.9) node[below=3pt] {$ t_1 $};
    \draw[thin] (0,2) -- (4,2); 
     \draw (0.8,2) node[below] {$v_0\in H$};
     \draw(2,2) node[above] {$z-\hbox{weak solution}$};
\fill[] (0,2) circle (0.5ex) node[below] {};
 \draw[-] (8,1.1) -- (8,0.9) node[below=3pt] {$ t_2 $};

        \draw[thick] (4,2) -- (8,2);
        \draw (6,2) node[above] {strong solution};
        \draw (4,2) node[below] {$v(t_1)\in H^1$}; 
 \draw (8,2) node[below] {$v(t_2)\in H_\eta^1$}; 
     \draw[line width=1mm, black] (8,2) -- (11, 2);
      \fill[] (11.5,2) circle (0.5ex) node[below] {};
       \fill[] (12,2) circle (0.5ex) node[below] {};
        \fill[] (12.5,2) circle (0.5ex) node[below] {};
      \draw (10, 2) node[above] {global strong solution}; 

\draw[thin, dashed] (0,3) -- (11,3);
\draw (1.1,3) node[above] {$\partial_z v_0\in L^\eta$}; 

\draw (11.5,3) node[above]   {$\partial_z v(t)\in L^\eta$};

        \fill[] (4,2) circle (0.4ex) node[below] {};

\fill[] (8,2) circle (0.4ex) node[below] {};

\end{tikzpicture}
\caption{From $z$-weak to global solution}\label{fig:ext}
\end{center}
\end{figure}

\end{proof}

\section{Some inequalities}\label{sec:ineq}
	
\begin{lemma}[Poincar\'e inequality for lateral vanishing trace] \label{lem:poincare} There exists a constant $C>0$ such that 
\begin{align*}
||v||_{L^2(\Omega)} &\leq C ||\nablah v||_{L^2(\Omega)} \quad \hbox{for} \quad v\in  L^2((-h,h),H^1_0(G)),\\
||\nablah v||_{L^2(\Omega)} &\leq C ||\Deltah v||_{L^2(\Omega)} \quad \hbox{for} \quad v\in  L^2((-h,h), H^2(G)\cap H_0^1(G)).
\end{align*}
\end{lemma}	
	\begin{proof}
		Note that by the classical $2D$-Poincar\'e inequality for some $C_{2D}>0$
		\begin{align*}
		||v(\cdot,z)||^2_{L^2(G)} \leq C_{2D}^2 ||\nablah v(\cdot, z)||^2_{L^2(G)} \quad \hbox{  for allmost every } z\in (-h,h), 
		\end{align*} 
		and integrating with respect to $z$ gives the first estimate with $C=C_{2D}\sqrt{2h}$.
		
		For the second inequality, note that  
		there is a $C_{2D}>0$ such that
		\begin{align*}
		||v(\cdot,z)||^2_{L^2(G)} \leq C_{2D}^2 ||\Deltah v(\cdot, z)||^2_{L^2(G)} \quad \hbox{  for allmost every } z\in (-h,h), 
		\end{align*} 
		and integrating with respect to $z$ gives $	||v||^2_{L^2(\Omega)} \leq C ||\Deltah v||^2_{L^2(\Omega)}.$
		Using this, Cauchy-Schwartz and Young's inequality yields
		\begin{align*}
\langle \nablah v, \nablah v \rangle_{L^2(\Omega)} =  \langle v, \Deltah v \rangle_{L^2(\Omega)} &\leq \frac{1}{2}\norm{v}^2 + \frac{1}{2}\norm{\Deltah v}^2 \leq\frac{(C+1)}{2}\norm{\Deltah v}^2. \qedhere
		\end{align*}	
	\end{proof}
	
	The following inequalities is helpful 
	for proving local \textit{a priori} estimates.  	
	\begin{lemma}[Tri-linear estimates]\label{lemma:lowreg}\ \\
		a) Let $f,g\in L^2((-h,h),H^1_0(G))$ and $h\in H^1((-h,h),L^2(G))$. Then
		\begin{align*}
		|\left< fg,h \right>|&\leq c \norm{\nablah f}^{1/2}_{L^2}\norm{f}^{1/2}_{L^2}
		\norm{\nablah g}^{1/2}_{L^2}\norm{g}^{1/2}_{L^2}
		\left(\|\partial_z h\|^{1/2}_{L^{2}}\|h\|^{1/2}_{L^{2}}+\|h\|_{L^{2}}\right).
		\end{align*}
		b) Let $f\in L^2(\Omega)$, $g\in H^1(\Omega)$ and $h\in L^2((-h,h),H^1_0(G))$. Then
		\begin{multline*}
		|\left< fg,h \right>|\\ \leq c \norm{f}_{L^2}  \norm{\nablah h}_{L^2}^{1/2}
		\norm{h}_{L^2}^{1/2} (\norm{g}_{L^2} +\norm{\partial_z g}_{L^2})^{1/2}
		(\norm{g}_{L^2} +\norm{\nablah g}_{L^2})^{1/2}.
		\end{multline*}
	\end{lemma}
	\begin{proof}
		a) We have
		\begin{align*}
		|\left< fg,h \right>_{\Omega}|&=\int_{-h}^h|\left< f(z)g(z),h(z) \right>_{G}|\d\! z\\
		&\leq \int_{-h}^h \|f(z)\|_{L^4(G)}  \|g(z)\|_{L^4(G)} \|h(z)\|_{L^2(G)}\d\! z\\
		&\leq \|h\|_{L^{\infty}((-h,h),L^2(G))}\int_{-h}^h \|f(z)\|_{L^4(G)}  \|g(z)\|_{L^4(G)}\d\! z\\
		&\leq \|h\|_{L^{\infty}((-h,h),L^2(G))}\|f\|_{L^{2}((-h,h),L^4(G))}\|g\|_{L^{2}((-h,h),L^4(G))}.
		\end{align*}
		By Ladyzhenskaya's inequality $\norm{f(z)}^2_{L^4(G)}\leq c\norm{\nablah f(z)}_{L^2(G)}\norm{f(z)}_{L^2(G)}$ we obtain 
		\begin{align*}
		\|f\|^2_{L^{2}((-h,h),L^4(G))}	=\int_{-h}^h \|f(z)\|_{L^4(G)}^2 \d\! z & \leq c\int_{-h}^h \norm{\nablah f(z)}_{L^2(G)}\norm{f(z)}_{L^2(G)}\d\! z\\
		&\leq c \norm{\nablah f}_{L^2(\Omega)}\norm{f}_{L^2(\Omega)},
		\end{align*}
		and by the Gagliardo-Nirenberg interpolation inequality 
		\begin{align*}
		\|h\|_{L^{\infty}((-h,h),L^2(G))}\leq c (\|\partial_z h\|^{1/2}_{L^{2}(\Omega)}\|h\|^{1/2}_{L^{2}(\Omega)}+\|h\|_{L^{2}(\Omega)}).
		\end{align*}
		Hence we get
		\begin{align*}
		|\left< fg,h \right>_{\Omega}|&\leq c \left(\|\partial_z h\|^{1/2}_{L^{2}}\|h\|^{1/2}_{L^{2}}+\|h\|_{L^{2}}\right)\norm{\nablah f}^{1/2}_{L^2}\norm{f}^{1/2}_{L^2}\norm{\nablah g}^{1/2}_{L^2}\norm{g}^{1/2}_{L^2}.
		\end{align*}
		For the proof of $b)$ see \cite[Lemma 2.1 (a)]{Saal}.
	\end{proof}

\begin{lemma}[Non-linear Gr\"onwall inequality, cf. \cite{Dragomir}]\label{lem:nonlinGron} 
	Let $x\colon [a,b]\rightarrow \R_+$ be a continuous function that satisfies the inequality:
	\begin{align*}
	   x(t) \leq M + \int_a^t \Psi(s)\omega(x(s)) \d\! s, \quad t\in [a,b],
	\end{align*}
	where $M\geq 0$, $\Psi\colon [a,b]\rightarrow \R_+$ is continuous and $w\colon \R_+ \rightarrow \R_+$ is continuous and monotone-increasing. 
	
	Then the estimate
	 	\begin{align*}
	 	x(t) \leq \Phi^{-1}\left(\Phi(M) + \int_a^t \Psi(s)  \d\! s\right), \quad t\in [a,b],
	 	\end{align*}
	 	holds, where $\Phi\colon \R \rightarrow \R$ is given by 
	 	\begin{align*}
	 	\Phi(u):=\int_{u_0}^u \frac{\d\! s}{w(s)}, \quad u\in \R.
	 	\end{align*}
\end{lemma}

\end{document}